\documentclass[12pt,a4paper]{article}
\usepackage[margin=1in]{geometry}
\usepackage[latin1]{inputenc}
\usepackage{amsmath}
\usepackage{amsthm}
\usepackage{amsfonts}
\usepackage{amssymb}
\usepackage[makeroom]{cancel}
\usepackage{graphicx}
\usepackage{epsfig}
\usepackage{bm}
\usepackage{fullpage}
\usepackage{color}
\usepackage{hyperref}
\usepackage[noblocks]{authblk}
\usepackage{caption,subcaption}
\usepackage{float}
\usepackage{wrapfig}
\usepackage{caption}
\usepackage{cite}
\usepackage{cleveref}
\usepackage{cases}
\usepackage{enumitem}
\usepackage{tikz}
\usepackage{ifthen}
\DeclareMathSizes{10}{10}{10}{10}

\hypersetup{%
    colorlinks=true,
        linkcolor=blue,
    filecolor=magenta,      
    urlcolor=violet,
    citecolor=blue,
}

\title{Existence and properties of solutions of extended play-type hysteresis model}
\author[1]{K. Mitra\footnote{email: \href{mailto:koondanibha.mitra@inria.fr}{koondanibha.mitra@inria.fr} }}
\affil[1]{INRIA Paris, 2 Rue Simone IFF, 75012 Paris, France}
\numberwithin{equation}{section}
\makeatletter
\DeclareMathSizes{\@xpt}{\@xpt}{6}{5}
\makeatother

\newcounter{Lmm}
\newcounter{Prop}
\newcounter{rmrk}
\newcounter{def}
\newcounter{assumption}
 \stepcounter{assumption}
 
 \newcounter{EUassumption}
 \stepcounter{EUassumption}
  \newcounter{EUXproperties}
 \stepcounter{EUXproperties}
 
\newtheorem{theorem}{Theorem}[section]
\newtheorem{lemma}{Lemma}[section]
\newtheorem{proposition}[Prop]{Proposition}
\newtheorem{corollary}[Lmm]{Corollary}
\newtheorem{remark}[rmrk]{Remark}
\newtheorem{definition}[def]{Definition}

\def \a  {\alpha}

\def \d  {\delta}
\def \D  {\Delta}

\def \vr  {\varrho}

\def \Om {\Omega}

\def \t  {\tau}

\def \del {\nabla}
\def \Df {\mathcal{D}}
\def \h  {{\sc H}}

\def \vfl  {{\bm{F}}}

\def \p  {\partial}
\def \N  {{\mathbb{N}}}
\def \R  {{\mathbb{R}}}

\def \Pim  {{p^{(i)}_c}}
\def \Pdr  {{p^{(d)}_c}}
\def \rim  {{\rho^{(i)}}}
\def \rdr  {{\rho^{(d)}}}

\def \sgn {{\rm sign}}

\def \W {{\bf {\cal W}}}

\def \vg {\bm{g}}

\newcommand{\norm}[1]{ \bigl\| #1 \bigr\|}

\usepackage{lipsum}

\let\OLDthebibliography\thebibliography
\renewcommand\thebibliography[1]{
  \OLDthebibliography{#1}
  \setlength{\parskip}{0pt}
  \setlength{\itemsep}{0pt plus 0.3ex}
}

\begin{document}
\maketitle

\begin{abstract}
This paper analyses the well-posedness and properties of the extended play-type model which was proposed in \cite{Kmitra2017} to  incorporate hysteresis in unsaturated flow through porous media. The model, when regularised, reduces to a nonlinear degenerate parabolic equation coupled with an ordinary differential equation. This has an interesting mathematical structure which, to our knowledge, still remains unexplored. The existence of solutions for the non-degenerate version of the model is shown using the Rothe's method. An equivalent to maximum principle is proven for the solutions. Existence of solutions for the degenerate case is then shown assuming that the initial condition is bounded away from the degenerate points. Finally, it is shown that if the solution for the unregularised case exists, then it is contained within physically consistent bounds. 
\end{abstract}
\section{Introduction}
The Richards equation is commonly used to model the flow of water and air through soil \cite{Bear1979,helmig1997multiphase}. It is obtained by combining the mass balance equation with the Darcy's law  \cite{Bear1979,helmig1997multiphase}. In dimensionless form it reads,
\begin{subequations}\label{Eq:ExPlayModel}
\begin{align}
\p_t S +\del\cdot[k(S)(\del p - \vg)]=0,\label{Eq:Richard}
\end{align}
where the water saturation $S$ and capillary pressure $p$ are the primary unknowns. The saturation $S$ is assumed to be bounded in $[0,1]$, $k:\R\to \R^+$ represents the relative permeability function and $\vg$ represents the gravitational acceleration assumed to be constant in our case. To close \eqref{Eq:Richard}, a relation between $S$ and $p$ is generally assumed. In this paper, we use the extended play-type hysteresis model (henceforth called the `EPH' model) \cite{Kmitra2017} for this purpose which equates $S$ and $p$ as
\begin{align}
p \in \tfrac{1}{2}(\Pdr(S)+\Pim(S))- \tfrac{1}{2}(\Pdr(S)-\Pim(S))\;\mathrm{sign}(\p_t \h(S) + \p_t p).\label{Eq:model}
\end{align}
\end{subequations}
 The functions $\Pim,\,\Pdr,\,\h:(0,1]\to \R$ are determined based on experiments \cite{Bear1979,helmig1997multiphase,van1980closed,brooks1966properties} and $\mathrm{sign}(\cdot)$ is the multivalued signum graph given by
\begin{align}
    \mathrm{sign}(\zeta)=
    \begin{cases}
    1 &\text{ for } \zeta>0,\\
    [-1,1] &\text{ for } \zeta=0,\\
    -1 &\text{ for } \zeta<0.
    \end{cases}\label{EUXSign}
\end{align} 
A detailed explanation of the EPH model \eqref{Eq:model} is given later. On top of being nonlinear, the model  \eqref{Eq:ExPlayModel} is also degenerate in the sense that $k(S)$ can vanish and $\Pim(S),\,\Pdr(S)$ explode as $S\to 0$.

To analyse the model \eqref{Eq:ExPlayModel}, we will study another class of problems, i.e.
\begin{subequations}\label{Eq:ParaXode}
\begin{align}
&\p_t u +\del\cdot \vfl(u,v)=\del\cdot[\Df(u,v)\del u]+ \Psi_1(u,v),\label{Eq:UVformulationRichard}\\
&\p_t v=\Psi_2(u,v),\label{Eq:UVformulationOde}
\end{align}
\end{subequations}
completed with suitable boundary and initial conditions. The equivalence of \eqref{Eq:ParaXode} with the regularised version of \eqref{Eq:ExPlayModel}  will be shown in \Cref{Sec:MathForm}. System \eqref{Eq:ParaXode} has an interesting mathematical structure as it consists of a nonlinear parabolic partial differential equation in a two-way coupling with an ordinary differential equation. This implies that standard techniques, such as the $L^1$ contraction \cite{otto1996l1} principle or the Kirchhoff transform \cite{alt1984nonstationary} can not directly be applied to this system. Moreover, the maximum principle can be violated for the system \eqref{Eq:ParaXode} which gives rise to the overshoot phenomenon discussed in \cite{VANDUIJN2018232,mitra2018wetting}.

\begin{figure}[h!]
\begin{subfigure}{.5\textwidth}
\includegraphics[scale=.35]{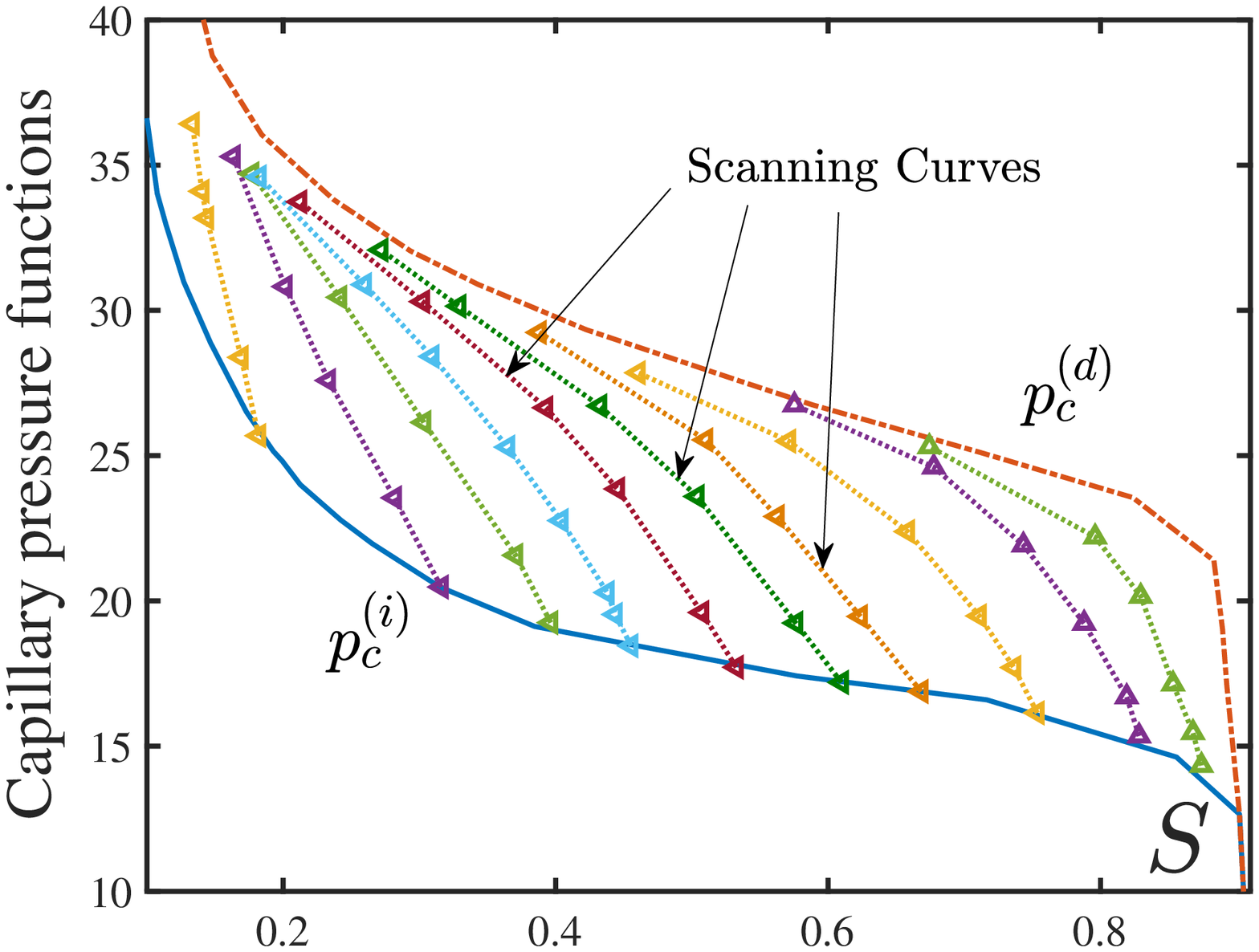}
\end{subfigure}
\begin{subfigure}{.5\textwidth}
\includegraphics[scale=.35]{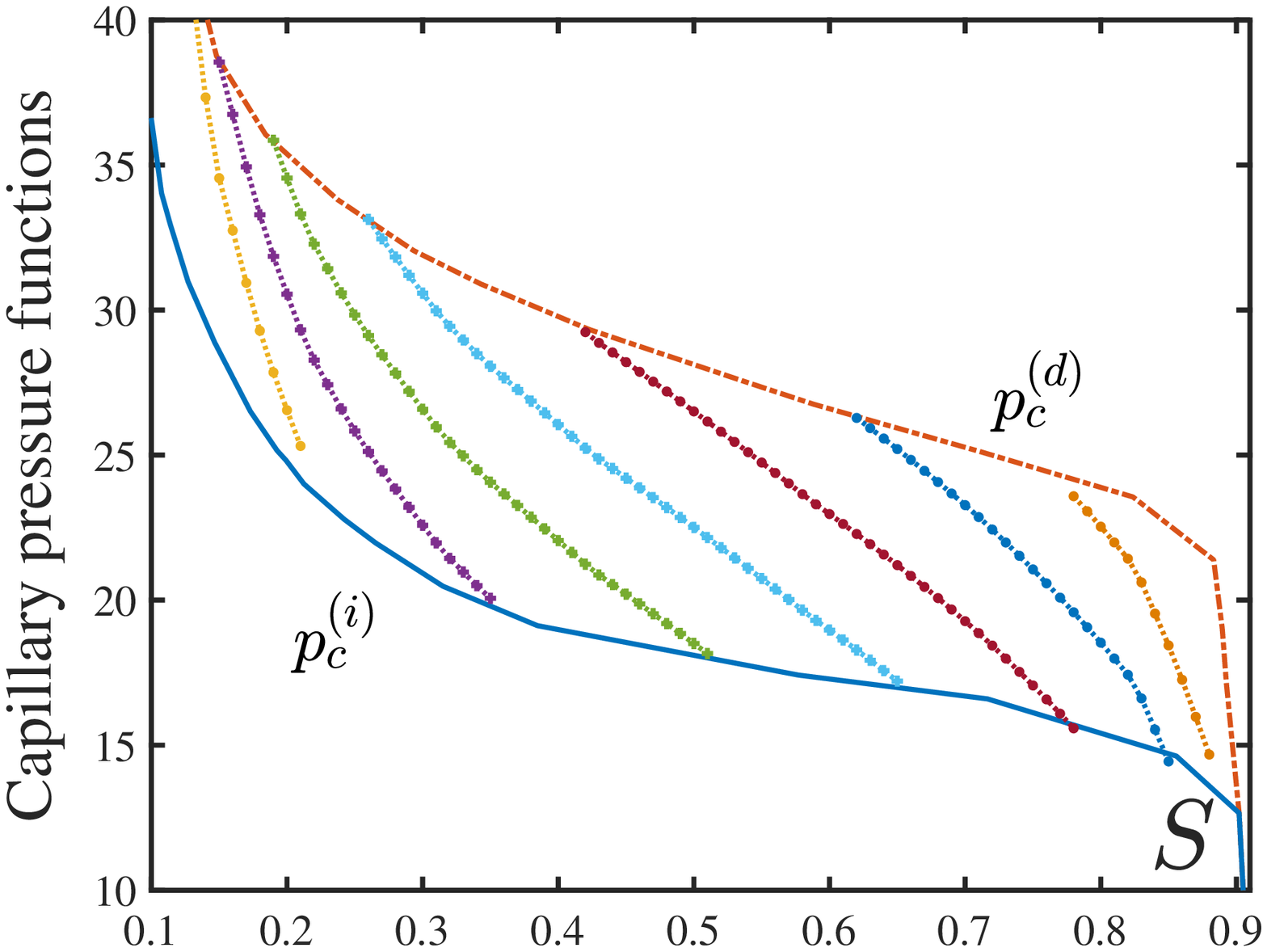}
\end{subfigure}
\caption{(left) Plots of $\Pim(S)$, $\Pdr(S)$ and scanning curves from imbibition experiments in \cite{Morrow_Harris}.  (right) Plot from \cite{Kmitra2017} showing the scanning curves  for the extended play-type hysteresis model \eqref{Eq:model} fitted from the left figure by tuning $\h(S)$. }
\label{fig:ThePcCurves}
\end{figure}

Further motivation for our analysis comes from establishing a physically consistent and well-posed model for hysteresis. Hysteresis refers to the phenomenon that if $p$ is measured in an imbibition experiment where $\p_t S>0$ at all times then it follows a imbibition pressure curve denoted by $\Pim(S)$. On the other hand, if $\p_t S<0$ at all times, then a different curve, i.e. drainage curve or $\Pdr(S)$, is followed. Curves intermediate to $\Pim$ and $\Pdr$ (where $\Pim(S)<p<\Pdr(S)$) are followed if the sign of $\p_t S$ changes from positive to negative or vice versa. These are referred to as the scanning curves, see \Cref{fig:ThePcCurves}.
This phenonmena was first documented in 1930 by Haines \cite{haines1930studies} and since then has been verfied by numerous experiments, \cite{Morrow_Harris,poulovassilis1970hysteresis,zhuang2017analysis}
being some notable examples. Several different classes of models are available that include hysteresis, Lenhard-Parker model \cite{parker1987parametric},  interfacial area models \cite{HASSANIZADEH1990169,niessner2008model},
and TCAT (thermodynamically constrained averaging theory) models \cite{miller2019nonhysteretic,miller2019nonhysteretic}  are examples of such. Review of hysteresis from a mathematical, modelling
and physical perspective can be found in \cite{schweizer2017hysteresis,Kmitra2017,miller2019nonhysteretic,MyThesis}. 

However, due to the complex nature of these models, effects of hysteresis are often ignored in many practical applications and $p$ is simply approximated by some weigthed average of $\Pim(S)$ and $\Pdr(S)$ \cite{Bear1979,helmig1997multiphase}, e.g.,
\begin{equation}
p=\dfrac{1}{2}(\Pdr(S) + \Pim(S)).\label{Eq:StandRelationship}
\end{equation}
Equation \eqref{Eq:Richard} along with \eqref{Eq:StandRelationship} constitutes a nonlinear diffusion problem with diffusivity $-\frac{1}{2}k(S)(\Pim'(S)+\Pdr'(S))$. The existence of weak solutions of such problems is studied in \cite{alt1983quasilinear,alt1984nonstationary} and uniqueness is shown in \cite{otto1996l1,carrillo1999uniqueness} using $L^1$ contraction principle. 
However, this simplification introduces significant errors, major examples being in trapping \cite{joekar2013trapping}, gravity fingering \cite{ratz2014hysteresis}, overshoots \cite{mitra2018wetting} and redistribution \cite{Raats_vD}. To account for hysteresis in a comparatively simple way, the play-type hysteresis model was proposed in \cite{Beliaev2001} based on thermodynamic considerations.  It reads,
\begin{equation}
p\in \tfrac{1}{2}(\Pdr(S)+\Pim(S))- \tfrac{1}{2}(\Pdr(S)-\Pim(S))\;\sgn(\p_t S).\label{Eq:playtype}
\end{equation}
Observe that, this gives $\p_t S\geq 0$ when $p=\Pim(S)$ and $\p_t S\leq 0$ when $p=\Pdr(S)$. However, $\Pim(S)<p<\Pdr(S)$ forces that $\p_t S=0$ which implies that the scanning curves for this model are vertical lines having constant saturation. Because of this simple structure, it can be considered to be the simplest possible model that addresses hysteresis.

The play-type hysteresis model has been studied extensively analytically due to its local and closed-form structure in contrast to other more complicated models such as the Lenhard-Parker model \cite{parker1987parametric}.  If the $\sgn$ graph in \eqref{Eq:playtype} is regularised, then together with \eqref{Eq:Richard}, it constitutes a nonlinear pseudo-parabolic equation for $S$ \cite{Cao2015688,lamacz2011well,SCHWEIZER20125594,mikelic2010global}. 
Existence results for such pseudo-parabolic equations can be found in \cite{bertsch2013pseudoparabolic,bohm1985diffusion,bohm1985nonlinear}. The existence of solutions of the regularised play-type hysteresis model with degenerate capillary pressure and permeability functions is shown in   \cite{SCHWEIZER20125594}. Existence for the two-phase case is discussed in \cite{koch2013}. For the constant relative permeability case, the existence of weak solutions for the unregularised play-type hysteresis model is shown in \cite{schweizer2007averaging}.  In  the same paper, an upscaled version of the play-type model is also  derived. Uniqueness is shown in \cite{Cao2015688} for the two-phase regularised case (i.e. with dynamic capillarity).
 In \cite{schweizer2012instability} it is shown that the play-type hysteresis model does not define an $L^1$-contraction. This leads to  unstable planar fronts.
Travelling wave solutions are investigated in \cite{VANDUIJN2018232,mitra2018wetting,el2018traveling,mitra2019fronts}. For mathematical analysis, in many cases the pseudo-parabolic system is split into an elliptic equation coupled with an ordinary differential equation \cite{SCHWEIZER20125594,lamacz2011well,bohm1985diffusion,bohm1985nonlinear,cao2015uniqueness}.

However, it was pointed out in \cite{Kmitra2017} that due to approximation of the scanning curves by vertical segments,  play-type hysteresis model makes certain physically inaccurate predictions. For example, it predicts that in many cases water will not redistribute when two columns having constant but different saturations are joined together. In \cite{mitra2018wetting} it is shown that the model predicts infinitely many interior maxima of saturation (called \emph{overshoots} in this context) for high enough injection rates through a long column. However, only finite number of overshoots are observed from experiments \cite{dicarlo2004experimental}. Moreover, it is well documented that the numerical methods incorporating the play-type model become unstable if the regularisation parameter is sent to zero \cite{Zhang2017,VANDUIJN2018232,schneider2018stable,brokate2012numerical}.  This motivated the extension of the play-type hysteresis model (EPH)  given by \eqref{Eq:model} in \cite{Kmitra2017}. An equivalent expression was derived in \cite[Eq. (35)]{Beliaev2001} using thermodynamic arguments. It was used in \cite{Kmitra2017} to cover all cases of horizontal redistribution and in \cite{mitra2018wetting} to explain the occurrence of finitely many overshoots, see also \cite{MyThesis}.

In this paper, we investigate the existence of weak solutions of the regularised EPH model and analyse the properties of its solutions.
An alternative form of \eqref{Eq:ExPlayModel} is proposed  in \Cref{Sec:MathForm} and mathematical preliminaries are stated. In \Cref{Sec:Existence}, existence of solutions for the model  given by \eqref{Eq:ParaXode} is proven using Rothe's method. 
In \Cref{Sec:BoundsandLimits}, it is shown that a maximum principle holds for the solutions under certain assumptions. This also gives the existence of solutions for the degenerate EPH model.
\Cref{Sec:EpsToZero} is dedicated to investigating the behavior of  the solutions when the regularisation  parameter is passed to 0. It is shown that if a limiting solution  $(S,p)$ exists then $p\in [\Pim(S),\Pdr(S)]$ almost everywhere.

\section{Mathematical formulation}\label{Sec:MathForm}
Let $\Om\subset\R^d$ be a bounded open domain with $\p\Om\in C^1$. The $L^2(\Om)$ inner product in this domain is denoted by $(\cdot,\cdot)$, whereas $\|\cdot\|$ and $\|\cdot\|_{p}$ denotes the $L^2(\Om)$ and $L^p(\Om)$ norms respectively for $1\leq p\leq \infty$. For any other space $V(\Om)$, the norm is denoted by $\|\cdot\|_V$. Further, $W^{k,p}$ denotes the Sobolev space containing functions that have up to $k^{\mathrm{th}}$ order derivatives in $L^p(\Om)$. In particular, $H^k(\Om):=W^{k,2}(\Om)$ and $H^k_0:=\{u\in H^k: u=0 \text{ on } \p\Om \text{ in a trace sense}\}$. Let $H^{-1}(\Om)$ denote the dual of $H^1_0$ and $\langle\cdot,\,\cdot\rangle$ the duality pairing of $H^1_0(\Om)$ with $H^{-1}(\Om)$.  The space of function having upto  $\ell^{\mathrm{th}}$ order space derivatives $\a$-H\"older continuous, will be referred to as $C^{\ell,\a}$ with $C^\ell:=C^{\ell,0}$.

Let $T>0$ represent a maximum time with $Q:=\Om\times (0,T]$. The Bochner space $L^p(0,T;X(\Om))$ represents the space of functions $u:Q\to \R$ having norm $\|u\|_{L^p(0,T;X(\Om))}:=\left (\int^T_0 \|u(t)\|_X^p dt\right )^{1/p} <\infty$.  Finally we introduce the space
$$
\W:=\{u\in L^2(0,T;H^1_0(\Om)): \p_t u \in L^2(0,T;H^{-1}(\Om))\}.
$$
 Following \cite{simon1986compact}, $\W$ is compactly embedded in $L^2(0,T;L^2(\Om))$, $\W\hookrightarrow L^2(0,T;L^2(\Om))$. Moreover, $\W$ is continuously embedded in $C(0,T;L^2(\Om))$ (space of time continuous functions $u$ with respect to the $L^2(\Om)$ norm, i.e. $\|u(t_n)-u(t)\|\to 0$ for any $t\in [0,T]$ and $t_p\to t$). 
 
The inequalities that are used repeatedly in our analysis include Cauchy-Schwarz inequality; Poincar$\acute{\text{e}}$ inequality, with $C_\Om$ denoting the constant appearing in it, i.e. $\|u\|\leq C_\Om\|\del u\|$ for $u\in H^1_0(\Om)$; Young's inequality, 
stating that for $a,b\in\R$  and $\sigma>0$ one has
\begin{equation}
ab\leq \frac{1}{2\sigma}a^2+ \frac{\sigma}{2}b^2, 
\end{equation}
and finally the discrete Gronwall's lemma, which states that if $\{y_n\}$, $\{f_n\}$ and $\{g_n\}$ are non-negative sequences and
\begin{equation*}
y_n \leq f_n  + \underset{0 \leq k< n}{\sum}g_k y_k, \quad \text{ for all }   n\geq 0,
\end{equation*}
then
\begin{equation*}
y_n \leq f_n  + \underset{0 \leq k< n}{\sum}f_kg_k \exp\left(\underset{k<j<n}{\sum}g_j\right), \quad  \text{ for all }   n\geq 0.
\end{equation*} 
In our notation $[\cdot]_+$ represents the positive part function, i.e. $[\zeta]_+:=\max\{\zeta,0\}$. Similarly, $[\zeta]_-:=\min\{\zeta,0\}$. Further, $C>0$ will denote a generic constant throughout the paper.

\subsection{Problem statement and assumptions}
For the most part, in this paper we study the system \eqref{Eq:ParaXode}. Completed with initial and boundary conditions, it is written as
\begin{numcases}{\hypertarget{link:Ps}{(\mathcal{P}s)}}
\p_t u + \del\cdot\vfl(u,v)=\del\cdot[\Df(u,v)\del u ]+\Psi_1(u,v)  & in $ Q$,\label{Eq:gde1}\\
\p_t v=\Psi_2(u,v)  & in $ \bar{Q}$,\label{Eq:gde2}\\
u(\cdot,0)=u_0(\cdot), \quad v(\cdot,0)=v_0(\cdot)  & in $\Om$, \label{Eq:IC}\\
u=0 & on $\p\Om\times[0,T]$. \label{Eq:BC}
\end{numcases}
The relation between $\hyperlink{link:Ps}{(\mathcal{P}s)}$ and EPH is established in \Cref{Sec:RelationEPHandPw}.

The properties of the functions $\Df$, $\Psi_{1\slash 2}$, $u_0$ and $v_0$ are as follows:
\begin{enumerate}[label=(A\theEUassumption)]
\item $\Df\in C^1(\R^2)$; $0<\Df_m\leq \Df(u,v)\leq \Df_M<\infty$ for $u,v\in \R$.  \label{ass:D}\stepcounter{EUassumption}
\item $\vfl:\R^2\to \R^d$ with the $j^{\mathrm{th}}$-component  $F_j\in C^1(\R^2)$ satisfying $|\vfl|\leq F_M$ for some $F_M>0$.  \label{ass:F}\stepcounter{EUassumption}
\item $|\Psi_j(u_1,v_1)-\Psi_j(u_2,v_2)|\leq \Psi_u|u_1-u_2|+\Psi_v|v_1-v_2|$ for $j\in\{1,2\}$ and constants $\Psi_u,\Psi_v>0$. Moreover, $\frac{\Psi_j(u,v)-\Psi_j(u,v_1)}{v-v_1}, \frac{\Psi_j(u,v)-\Psi_j(u_1,v)}{u-u_1}\leq 0$ for all $u,u_1,v,v_1\in \R$. \label{ass:Psi}\stepcounter{EUassumption}
\item  $u_0\in L^\infty(\Om)$ and $v_0\in H^1(\Om)\cap L^\infty(\Om)$, such that $\Psi_2(0,v_0)=0$ at $\p\Om$. \label{ass:Ini}
\stepcounter{EUassumption}
\end{enumerate}
Observe that, \eqref{Eq:gde2} combined with  \ref{ass:Psi}--\ref{ass:Ini} imply that $v(t)$ restricted to $\p\Om$ remains unchanged for all $t>0$. The weak solution of $\hyperlink{link:Ps}{(\mathcal{P}s)}$ is defined as
\begin{definition}[Weak solution of $\hyperlink{link:Ps}{(\mathcal{P}s)}$]
The pair $(u,v)$ with $u\in \W$ and $v \in H^1(0,T;L^2(\Om))\cap L^2(0,T;H^1(\Om))$ is a weak solution of $\hyperlink{link:Ps}{(\mathcal{P}s)}$ if $u(0)=u_0$, $v(0)=v_0$ and it satisfies for all $\phi\in L^2(0,T;H^1_0(\Om))$ and $\xi\in L^2(Q)$
\begin{subequations}
\begin{numcases}{\hypertarget{link:Pw}{(\mathcal{P}w)}}
\smallint_0^T\langle \p_t u, \phi\rangle + \smallint_0^T(\Df(u,v)\del u, \del \phi)=\smallint_0^T(\vfl(u,v),\del \phi)+\smallint_0^T (\Psi_1(u,v),\phi);\label{Eq:weak1}\\
\smallint_0^T (\p_t v,\xi)= \smallint_0^T (\Psi_2(u,v),\xi).\label{Eq:weak2}
\end{numcases}
\end{subequations}
\label{def:WeakSol}
\end{definition}

\noindent
Due to the continuous embedding of $\W$ in $C(0,T;L^2(\Om))$, $u(0)$ and $v(0)$ are well-defined.

\begin{remark}[Boundary conditions]
For simplicity, a zero Dirichlet condition has been assumed at the boundary for our current analysis. Nevertheless, Definition \ref{def:WeakSol} can be generalised to include Dirichlet, Neumann, and mixed type boundary conditions. 
\end{remark}

\begin{remark}[Assumptions]
The condition in \ref{ass:Psi} that $\Psi_j$ is decreasing with respect to both the variables $u$ and $v$ is not required for proving the existence of the weak solutions. It is used  in \Cref{Sec:BoundsandLimits} to prove that the solutions are bounded in $L^\infty$. Similarly $u_0,v_0\in L^\infty(\Om)$ is only used for proving the $L^\infty$ bound. For proving the existence result in  \Cref{Theo:Existence}, assuming $u_0\in L^2(\Om)$ and $v_0\in H^1(\Om)$ is sufficient.
\end{remark}

\subsection{Relation between the regularised EPH model and $(\mathcal{P}s)$}\label{Sec:RelationEPHandPw}
Although \eqref{Eq:ExPlayModel} is closer to the expressions of the hysteresis models used in practice, $\hyperlink{link:Ps}{(\mathcal{P}s)}$ is more convenient to analyse mathematically. We show below that the EPH model with the $\sgn(\cdot)$ graph regularised is a particular  case of $\hyperlink{link:Ps}{(\mathcal{P}s)}$. To be more precise, we start with assuming the following properties of the functions $\Pim$, $\Pdr$ and $k$ used in \eqref{Eq:ExPlayModel}, \eqref{Eq:StandRelationship} and \eqref{Eq:playtype}, which are consistent with the data obtained from experiments 
 \cite{helmig1997multiphase,Morrow_Harris,poulovassilis1970hysteresis}. 
\begin{enumerate}[label=(P\theEUXproperties)]
 \setlength\itemsep{-0.8em}
\item $\Pim,\, \Pdr\in C^1((0,1))$; $\Pim'(S),\,\Pdr'(S)<0$; $\Pdr(S)>\Pim(S)$ for all $S\in(0,1)$; $\Pim(1)=\Pdr(1)$ and $$\lim\limits_{S\searrow 0}\Pim(S)=\lim\limits_{S\searrow 0}\Pdr(S)=\lim\limits_{S\nearrow 1}\left [-\Pim'(S)\right ]=\lim\limits_{S\nearrow 1}\left [-\Pdr'(S)\right ]=\infty.$$\stepcounter{EUXproperties}\label{prop:Pc}
\item $k\in C^1(\R)$; $k(S)=k(0)\geq 0$ for $S\leq 0$, $k(S)=k(1)$ for $S\geq 1$ and $k(0)<k(S)<k(1)$ for $0<S<1$.\stepcounter{EUXproperties}\label{prop:K}
\end{enumerate}

The set of equations \eqref{Eq:ExPlayModel} cannot directly be reduced to the standard weak formulation used for partial differential equations. Thus, we consider
 an alternative expression to \eqref{Eq:model} representing the EPH model. Completed with suitable boundary and initial conditions, the model reads
\begin{numcases}{\hypertarget{link:EPH}{\mathrm{(EPH)}}}
\p_t S=\del\cdot[k(S)(\del p-\bm{g})]  & in $Q$,\hspace{2em}\label{Eq:Rgde1}\\
p\in \tfrac{1}{2}(\Pdr(S) + \Pim(S))-\tfrac{1}{2}(\Pdr(S) - \Pim(S))\;\sgn\left[\p_t (S + b(p))\right]  &in $ \bar{Q}$,\hspace{2em}\label{Eq:Rgde2}\\
S(\cdot,0)=S_0(\cdot),\quad p(\cdot,0)=p_0(\cdot)   &in $\Om$,\hspace{2em}\label{Eq:RIC}\\
p=0 &\hspace{-5em}on $\p\Om\times[0,T]$.\hspace{2em} \label{Eq:RBC}
\end{numcases}

Observe that, for relation \eqref{Eq:Rgde2} the scanning curves are given by $S+b(p)=\textit{constant}$, instead of $\h(S)+p=\textit{constant}$ as used in \cite{Kmitra2017}. Moreover, if $p=\Pim(S)$ in some open subset of $\Om$, then from \eqref{EUXSign} and \eqref{Eq:Rgde2}, $$\p_t S + \p_t b(\Pim(S))\geq 0 \quad \text{ or }\quad (1+ b'(\Pim(S))\,\Pim'(S))\, \p_t S\geq 0.$$ The directionality imposed by hysteresis then demands that $\p_t S\geq 0$ in this case. Hence, for consistency $1+ b'(\Pim(S))\,\Pim'(S)>0$ has to be satisfied. Similar result holds if $p=\Pdr(S)$ in some open subset of $\Om$. Combining these observations, we assume that
\begin{enumerate}[label=(P\theEUXproperties)]
 \setlength\itemsep{0.5em}
\item $b\in C^1(\R)$ with $b(0)=0$ 
 and
\begin{equation}
0< b'(p_c^{(j)}(S))<-\frac{1}{{p_c^{(j)}}'(S)} \text{ for all } 0<S<1 \text{ and } j\in \{i,d\}.\label{Eq:BoundOfbprime}
\end{equation}\stepcounter{EUXproperties}\label{prop:b}
\end{enumerate}
Here, \eqref{Eq:BoundOfbprime} is the consistency criterion, a counterpart of which was stated in \cite[Eq. (2.7)]{Kmitra2017} for $\h(S)$.

\begin{remark}[Consistency of the scanning curves]
The inequality \eqref{Eq:BoundOfbprime} also guarantees that  any scanning curve passing through $(S_1,p_1)$ for arbitrary $S_1\in (0,1)$ and $p_1\in (\Pim(S_1),\Pdr(S_1))$ intersects $\Pim$ at some $S_i<1$ and $\Pdr$ at some $S_d\in (0,S_i)$. 
\end{remark}

For the initial and boundary conditions we assume:
\begin{enumerate}[label=(P\theEUXproperties)]
 \setlength\itemsep{0.5em}
\item $S_0\in H^1(\Om)$  and $p_0\in H^1_0(\Om)$. Moreover, an $\epsilon>0$ exists such that $\epsilon\leq S_0\leq1-\epsilon$ and $\Pim(S_0)\leq p_0 \leq \Pdr(S_0)$ a.e. in $\Om$.
\label{prop:Ini}\stepcounter{EUXproperties}
\end{enumerate}
The condition $\Pim(S_0)\leq p_0 \leq \Pdr(S_0)$ comes from the physical constraint that $p_0$ stays intermediate to $\Pim$ and $\Pdr$ when only hysteretic effects are considered.

\begin{remark} [Degeneracy and physical bounds]
If $\lim\limits_{S\searrow 0} p_c^{(j)}(S)=\infty$ for $j\in\{i,d\}$ or $k(0)=0$, then $S=0$ at any interior point in the domain makes the problem degenerate since \eqref{Eq:Rgde1} looses its parabolicity. Similarly, the problem becomes degenerate at $S=1$ since $\Pim'(1)=\Pdr'(1)=-\infty$. Moreover, $S$ must satisfy the physical bound $0\leq S\leq 1$, otherwise $\Pim(S)$ and $\Pdr(S)$ become ill defined. Treating the degeneracy and proving the physical bounds pose extra challenges in considering the problem \hyperlink{link:EPH}{(EPH)} compared to $\hyperlink{link:Ps}{(\mathcal{P}s)}$.
\end{remark}

Having stated the properties of the associated functions, we now show the equivalence of the regularised \hyperlink{link:EPH}{(EPH)} model and $\hyperlink{link:Ps}{(\mathcal{P}s)}$. For this purpose,
the following transformations are introduced:
\begin{equation}
u:=b(p)=\int_0^p b'(\varrho)d\vr\quad \text{ and }\quad  v:=S+b(p)=S+u.\label{def:uv}
\end{equation}
We recast \hyperlink{link:EPH}{(EPH)} in terms of $u$ and $v$. Since $\lim_{S\searrow 0} \Pdr(S)=\infty$ we get integrating \eqref{Eq:BoundOfbprime},
\begin{align}\label{eq:UmVmLeq1}
b(\infty)-b(\Pdr(1))=\int^\infty_{\Pdr(1)} b'(p)\,dp\leq -\int^\infty_{\Pdr(1)} \frac{dp}{\Pdr'(\Pdr^{-1}(p))}\leq 1.
\end{align} 
Note that, if $S\to 0$ then $v\to u$. From these observations, we define the terms $U_m,V_m,U_M,V_M$ that will become important later;
\begin{equation}
U_M=V_m=b(\infty), \quad U_m=b(\Pim(1))=b(\Pdr(1)),\quad V_M=1+U_m>V_m.
\end{equation}
The $V_M>V_m$ inequality follows from \eqref{eq:UmVmLeq1}. Next, we express $p=\Pim(S)$ in terms of a relation between $u$ and $v$. From \eqref{def:uv}, $p=\Pim(S)$ implies $v=v_{i}(u):=(\Pim)^{-1}(b^{-1}(u)) +u$. Recalling \ref{prop:b}, ${v_{i}}'(u)=\frac{1}{b'(\Pim(S))\Pim'(S)} + 1<0$ where $S=(\Pim)^{-1}(b^{-1}(u))$. Hence, the inverse of $v_{i}(u)$ exists. Let it be denoted by $\rim(u)$. In a similar way $\rdr(u)$ is defined. The definitions can alternatively be summarized into
\begin{equation}
\rim:=((\Pim)^{-1}\circ b^{-1}+1)^{-1},\qquad \rdr:=((\Pdr)^{-1}\circ b^{-1}+1)^{-1}.\label{def:Rimdr}
\end{equation}
From \ref{prop:Pc} and \ref{prop:b}, one immediately obtains
\begin{enumerate}[label=(P\theEUXproperties)]
 \setlength\itemsep{0.5em}
\item There exists a constant $M_\rho>0$ such that $-M_\rho<\rim',\,\rdr'<0$ for $v\in[V_m,V_M]$. Further, $\rdr(v)>\rim(v)$ for all $v\in (V_m,V_M)$; $\rim(V_m)=\rdr(V_m)=U_M$ and $\rim(V_M)=\rdr(V_M)=U_m$. \stepcounter{EUXproperties}\label{prop:Rc}
\end{enumerate}
\Cref{fig:SpuvPlane} (right) plots $\rim$ and $\rdr$ curves calculated from realistic $\Pim$ and $\Pdr$ curves shown in \Cref{fig:SpuvPlane} (left). By definition, $p=\Pim(S)$ iff $u=\rim(v)$, and $p=\Pdr(S)$ iff $u=\rdr(v)$. Furthermore, $\p_t v=0$ implies $\Pim(S)\leq p \leq \Pdr(S)$ which is same as having $\rim(v)\leq u \leq \rdr(v)$.
Thus, an equivalent expression to \eqref{Eq:Rgde2} is 
\begin{align}
u\in \tfrac{1}{2}(\rdr(v) + \rim(v)) -\tfrac{1}{2}(\rdr(v) - \rim(v))\; \sgn(\p_t v). \label{Eq:X}
\end{align}

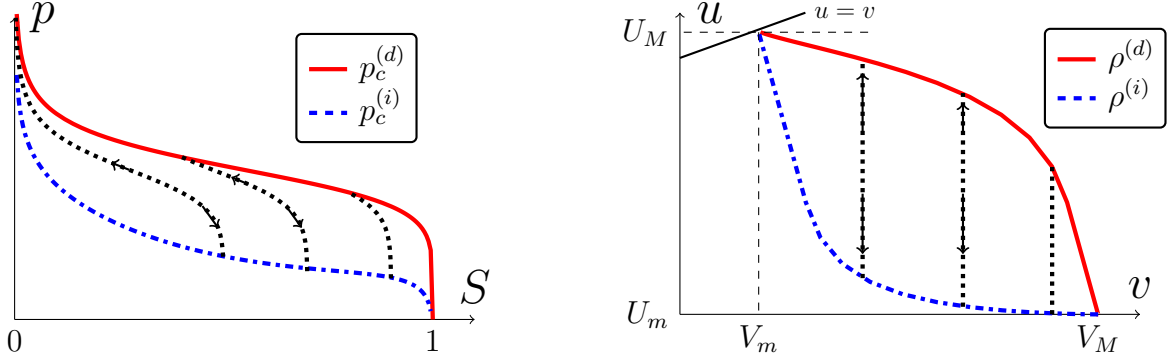
\begin{figure}
\begin{subfigure}{.5\textwidth}
\begin{tikzpicture}[xscale=5.5,yscale=4,domain=0:1,samples=100]

\draw[->] (0,0.2)--(0,1.2) node[right, scale=1.5] {$p$};
\draw[->] (0,0.2)--(1.1,0.2) node[above,scale=1.5] {$S$};
\node[scale=1,below] at (1,0.2) {1};
\node[scale=1,below] at (0,0.2) {0};

\draw[ultra thick, red](0.005,1.2126)--(0.01,1.1281)--(0.015,1.0813)--(0.02,1.0491)--(0.025,1.0247)--(0.03,1.0052)--(0.035,0.98886)--(0.04,0.97488)--(0.045,0.96267)--(0.05,0.95183)--(0.055,0.94209)--(0.06,0.93325)--(0.065,0.92515)--(0.07,0.91769)--(0.075,0.91076)--(0.08,0.9043)--(0.085,0.89824)--(0.09,0.89255)--(0.095,0.88717)--(0.1,0.88207)--(0.12,0.86398)--(0.14,0.84867)--(0.16,0.83537)--(0.18,0.82357)--(0.2,0.81292)--(0.22,0.80319)--(0.24,0.79421)--(0.26,0.78584)--(0.28,0.77797)--(0.3,0.77053)--(0.32,0.76346)--(0.34,0.75668)--(0.36,0.75017)--(0.38,0.74388)--(0.4,0.73778)--(0.42,0.73184)--(0.44,0.72604)--(0.46,0.72034)--(0.48,0.71474)--(0.5,0.70921)--(0.52,0.70372)--(0.54,0.69828)--(0.56,0.69284)--(0.58,0.68741)--(0.6,0.68195)--(0.62,0.67646)--(0.64,0.6709)--(0.66,0.66526)--(0.68,0.65952)--(0.7,0.65364)--(0.72,0.6476)--(0.74,0.64136)--(0.76,0.63486)--(0.78,0.62807)--(0.8,0.62091)--(0.82,0.61328)--(0.84,0.60508)--(0.86,0.59614)--(0.88,0.58623)--(0.9,0.57498)--(0.905,0.57191)--(0.91,0.5687)--(0.915,0.56534)--(0.92,0.56183)--(0.925,0.55813)--(0.93,0.55422)--(0.935,0.55008)--(0.94,0.54567)--(0.945,0.54094)--(0.95,0.53584)--(0.955,0.53028)--(0.96,0.52417)--(0.965,0.51737)--(0.97,0.50967)--(0.975,0.50075)--(0.98,0.49011)--(0.985,0.47679)--(0.99,0.45873)--(0.995,0.42958)--(1,0.2);

\draw[blue, ultra thick, dashdotted] (0.005,1.0106)--(0.01,0.92415)--(0.015,0.87536)--(0.02,0.84128)--(0.025,0.81499)--(0.03,0.79353)--(0.035,0.77535)--(0.04,0.75952)--(0.045,0.74548)--(0.05,0.73283)--(0.055,0.7213)--(0.06,0.71069)--(0.065,0.70084)--(0.07,0.69165)--(0.075,0.68301)--(0.08,0.67486)--(0.085,0.66713)--(0.09,0.65979)--(0.095,0.65278)--(0.1,0.64607)--(0.105,0.63963)--(0.11,0.63345)--(0.115,0.62749)--(0.12,0.62174)--(0.125,0.61618)--(0.13,0.61079)--(0.135,0.60558)--(0.14,0.60051)--(0.145,0.59559)--(0.15,0.59081)--(0.155,0.58615)--(0.16,0.58161)--(0.165,0.57719)--(0.17,0.57287)--(0.175,0.56865)--(0.18,0.56453)--(0.185,0.5605)--(0.19,0.55656)--(0.195,0.5527)--(0.2,0.54892)--(0.22,0.53455)--(0.24,0.52125)--(0.26,0.50888)--(0.28,0.49733)--(0.3,0.48653)--(0.32,0.47642)--(0.34,0.46692)--(0.36,0.45801)--(0.38,0.44964)--(0.4,0.44178)--(0.42,0.4344)--(0.44,0.42748)--(0.46,0.42098)--(0.48,0.4149)--(0.5,0.40921)--(0.52,0.40388)--(0.54,0.39892)--(0.56,0.39428)--(0.58,0.38997)--(0.6,0.38595)--(0.62,0.38222)--(0.64,0.37874)--(0.66,0.3755)--(0.68,0.37248)--(0.7,0.36964)--(0.72,0.36696)--(0.74,0.3644)--(0.76,0.3619)--(0.78,0.35943)--(0.8,0.35691)--(0.82,0.35424)--(0.84,0.35132)--(0.86,0.34798)--(0.88,0.34399)--(0.9,0.33898)--(0.905,0.33752)--(0.91,0.33594)--(0.915,0.33423)--(0.92,0.33239)--(0.925,0.33038)--(0.93,0.32818)--(0.935,0.32577)--(0.94,0.32311)--(0.945,0.32015)--(0.95,0.31684)--(0.955,0.31309)--(0.96,0.30881)--(0.965,0.30386)--(0.97,0.29803)--(0.975,0.291)--(0.98,0.28227)--(0.985,0.27088)--(0.99,0.25477)--(0.995,0.22759);

\draw[ultra thick,dotted] (0.49855,0.408)--(0.49735,0.432)--(0.4953,0.456)--(0.49195,0.48)--(0.48663,0.504)--(0.47846,0.528)--(0.46636,0.552)--(0.44915,0.576)--(0.4258,0.6)--(0.39575,0.624)--(0.35935,0.648)--(0.31803,0.672)--(0.27419,0.696)--(0.23061,0.72)--(0.18981,0.744)--(0.15353,0.768)--(0.12258,0.792)--(0.097009,0.816)--(0.076374,0.84)--(0.059985,0.864)--(0.047104,0.888)--(0.037039,0.912)--(0.029197,0.936)--(0.02309,0.96)--(0.018328,0.984)--(0.014607,1.008)--(0.011691,1.032)--(0.0093966,1.056)--(0.0075848,1.08)--(0.0061483,1.104)--(0.0050045,1.128)--(0.00409,1.152)--(0.0033558,1.176)--(0.002764,1.2);

\draw[->,thick] (0.44915,0.576)--(0.48663,0.504);
\draw[->,thick] (0.27419,0.696)--(0.23061,0.72);

\draw[ultra thick,dotted] (0.69962,0.36)--(0.69924,0.384)--(0.69855,0.408)--(0.69735,0.432)--(0.6953,0.456)--(0.69195,0.48)--(0.68663,0.504)--(0.67846,0.528)--(0.66636,0.552)--(0.64915,0.576)--(0.6258,0.6)--(0.59575,0.624)--(0.55935,0.648)--(0.51803,0.672)--(0.47419,0.696)--(0.43061,0.72)--(0.38981,0.744);

\draw[->,thick] (0.64915,0.576)--(0.68663,0.504);
\draw[->,thick] (0.55935,0.648)--(0.51803,0.672);

\draw[ultra thick,dotted] (0.89982,0.336)--(0.89962,0.36)--(0.89924,0.384)--(0.89855,0.408)--(0.89735,0.432)--(0.8953,0.456)--(0.89195,0.48)--(0.88663,0.504)--(0.87846,0.528)--(0.86636,0.552)--(0.84915,0.576)--(0.8258,0.6)--(0.79575,0.624);


\node[draw=black,thick,rounded corners=2pt,below left=2mm, scale=1] at (1,1.2){%
\begin{tabular}{@{}r@{ }l@{}}
 \raisebox{2pt}{\tikz{\draw[ultra thick, red, domain=0.05:0.06] (0,0) -- (5mm,0);}}&$\Pdr$\\
  \raisebox{2pt}{\tikz{\draw[dashed,ultra thick, blue, domain=0.05:0.06] (0,0) -- (5mm,0);}}&$\Pim$
\end{tabular}};
\end{tikzpicture}
\end{subfigure}
\begin{subfigure}{.5\textwidth}
\begin{tikzpicture}[xscale=5.5,yscale=4,domain=0:1,samples=100]

\draw[->] (0,0)--(0,1) node[right, scale=1.5] {$u$};
\draw[->] (0,0)--(1.1,0) node[above,scale=1.5] {$v$};
\node[scale=1,left] at (0,0) {$U_m$};

\draw[ultra thick, red] (1,0)--(1,1.4089e-16)--(1,9.417e-13)--(1,7.0296e-11)--(1,1.4999e-09)--(1,1.6107e-08)--(1,1.1204e-07)--(1,5.7752e-07)--(1,2.3906e-06)--(1,8.3691e-06)--(0.99999,2.5672e-05)--(0.99999,7.0762e-05)--(0.99996,0.00017856)--(0.99992,0.00041836)--(0.99982,0.00092013)--(0.99962,0.0019161)--(0.99924,0.0038041)--(0.99855,0.0072397)--(0.99735,0.013265)--(0.9953,0.023476)--(0.99195,0.040225)--(0.98663,0.066826)--(0.97846,0.10768)--(0.96636,0.16819)--(0.94915,0.25423)--(0.9258,0.37101)--(0.88956,0.49029)--(0.83589,0.58598)--(0.7684,0.66167)--(0.69258,0.72103)--(0.61469,0.76734)--(0.5403,0.80341)--(0.47335,0.83151)--(0.41584,0.85343)--(0.36813,0.87058)--(0.32956,0.88404)--(0.2989,0.89464)--(0.27481,0.90302)--(0.25601,0.90967)--(0.24139,0.91498)--(0.23002,0.91922)--(0.22118,0.92264)--(0.21429,0.9254)--(0.20891,0.92763)--(0.20468,0.92945)--(0.20136,0.93093)--(0.19873,0.93215)--(0.19664,0.93315)--(0.19497,0.93397)--(0.19364,0.93466)--(0.19257,0.93523);

\draw[blue, ultra thick, dashdotted] (1,0)--(1,1.4089e-16)--(0.99999,9.417e-13)--(0.99996,7.0296e-11)--(0.9999,1.4999e-09)--(0.99978,1.6107e-08)--(0.99953,1.1204e-07)--(0.99903,5.7752e-07)--(0.99811,2.3906e-06)--(0.99644,8.3691e-06)--(0.99356,2.5672e-05)--(0.98874,7.0762e-05)--(0.98091,0.00017856)--(0.96864,0.00041836)--(0.95003,0.00092013)--(0.92288,0.0019161)--(0.88495,0.0038041)--(0.83451,0.0072397)--(0.7712,0.013265)--(0.69674,0.023476)--(0.61517,0.040225)--(0.53236,0.066826)--(0.45473,0.10768)--(0.38795,0.16819)--(0.33605,0.25423)--(0.30114,0.37101)--(0.27729,0.49029)--(0.25815,0.58598)--(0.24301,0.66167)--(0.23114,0.72103)--(0.22188,0.76734)--(0.21466,0.80341)--(0.20904,0.83151)--(0.20466,0.85343)--(0.20123,0.87058)--(0.19854,0.88404)--(0.19642,0.89464)--(0.19474,0.90302)--(0.19341,0.90967)--(0.19235,0.91498)--(0.1915,0.91922)--(0.19082,0.92264)--(0.19027,0.9254)--(0.18982,0.92763)--(0.18946,0.92945)--(0.18916,0.93093)--(0.18891,0.93215)--(0.18871,0.93315)--(0.18855,0.93397)--(0.18841,0.93466)--(0.1883,0.93523);

\draw[dashed] (0.1883,0.93523)--(.1883,0) node[scale=1,below] {$V_m$};
\node[scale=1,below] at (1,0) {$V_M$};
\draw[dashed] (.45,0.93523)--(0,0.93523) node[scale=1,left] {$U_M$};
\draw[thick] (0,0.85)--(.3,1) node[scale=0.8,right] {$u=v$};

\draw[ultra thick,dotted] (0.43683,0.120)--(.4368,.843);
\draw[thick,->] (0.43683,0.60)--(0.43683,0.8);
\draw[thick,->] (0.43683,0.4)--(0.43683,0.2);

\draw[ultra thick,dotted] (0.6768,0.025)--(0.6768,.7335);
\draw[thick,->] (0.6768,0.4)--(0.6768,0.2);
\draw[thick,->] (0.6768,0.6)--(0.6768,0.7);

\draw[ultra thick,dotted] (0.89,0)--(0.89,.5);

\node[draw=black,thick,rounded corners=2pt,below left=2mm, scale=1] at (1.2,1){%
\begin{tabular}{@{}r@{ }l@{}}
 \raisebox{2pt}{\tikz{\draw[ultra thick, red, domain=0.05:0.06] (0,0) -- (5mm,0);}}&$\rdr$\\
  \raisebox{2pt}{\tikz{\draw[dashed,ultra thick, blue, domain=0.05:0.06] (0,0) -- (5mm,0);}}&$\rim$
\end{tabular}};
\end{tikzpicture}
\end{subfigure}
\caption{(left) Realistic $\Pim$ and $\Pdr$ curves  in the $S$--$p$ plane calculated using the van Genuchten model \cite{van1980closed}; (right) corresponding $\rim$ and $\rdr$ curves in the $v$--$u$ plain. The $b(p)$ used here is such that $b'(p)=-\frac{1}{2}\max\left(\frac{1}{\Pim'(\Pim^{-1}(p))},\, \frac{1}{\Pdr'(\Pdr^{-1}(p))}\right)$. The black dotted lines correspond to the scanning curves with respect to this choice. In particular, curves for $S+b(p)=0.5,\;0.7,\;0.9$ are shown. The values $U_m$, $U_M$, $V_m$ and $V_M$ are marked in the (right) figure.}\label{fig:SpuvPlane}
\end{figure}

Since $\sgn(\cdot)$ is not single-valued, we regularise  \eqref{Eq:X} using the expression
\begin{align}
u\in \tfrac{1}{2}(\rdr(v) + \rim(v)) -\tfrac{1}{2}(\rdr(v) - \rim(v))\; \sgn(\p_t v)-\tau \p_t v,\label{Eq:Xeps}
\end{align}
where $\t>0$ is a regularisation parameter.
 This approach has been used in \cite{lamacz2011well,cao2015uniqueness,ratz2014hysteresis}.
The right most term in \eqref{Eq:Xeps} also has physical significance, since, it gives rise to the dynamic capillarity phenomenon in porous media \cite{Cao2015688,hassanizadeh1993thermodynamic,SCHWEIZER20125594}.
Moreover, expression \eqref{Eq:Xeps} is thermodynamically consistent as it leads to entropy generation as is shown in \cite{Beliaev2001}, see specifically equations (28) and (35).
Since the function $\frac{1}{2}(\rdr(v) - \rim(v))\,\sgn(\zeta) + \tau \zeta$ is increasing with respect to $\zeta$, expression \eqref{Eq:Xeps} can be inverted, yielding the relation \cite{cao2015uniqueness,Beliaev2001,beliaev2001analysis}
\begin{equation}
\p_t v=\Phi_\tau(u,v):= \frac{1}{\tau}
\begin{cases}
\rdr(v)-u &\text{ when } u> \rdr(v),\\
0 &\text{ when } u\in [\rim(v),\rdr(v)],\\
\rim(v)-u &\text{ when } u< \rim(v).
\end{cases}\label{Eq:DefPhi_eps}
\end{equation}
Setting  in \hyperlink{link:EPH}{(EPH)}
\begin{equation}
\Df(u,v)=\frac{k(v-u)}{b'(b^{-1}(u))},\; \vfl(u,v)=\bm{g}k(v-u),\;\Psi_1(u,v)=\Psi_2(u,v)=\Phi_\tau(u,v),\label{Eq:DefkPsi12}
\end{equation}
 we recover $\hyperlink{link:Ps}{(\mathcal{P}s)}$. It is straightforward to verify from \ref{prop:Rc} that $\Psi_1$ and $\Psi_2$, defined as in \eqref{Eq:DefkPsi12}, satisfy assumption \ref{ass:Psi}. Similarly $u_0$ and $v_0$ are consistent with \ref{ass:Ini} when defined from a $(S_0,p_0)$ pair satisfying \ref{prop:Ini}. Furthermore, $U_m<u_0<U_M$ and $V_m<v_0<V_M$ in $\Om$. However, to show that $\Df(u,v)$ defined in \eqref{Eq:DefkPsi12}, satisfies assumption \ref{ass:D} ($\Df\geq \Df_M>0$) we need to  show that $S=v-u$ is bounded away from zero. This is done in \Cref{Sec:BoundsandLimits}.
 
 Based on our discussion so far, we define the weak solution of the \hyperlink{link:EPH}{(EPH)} model for $\tau>0$ as
 
 \begin{definition}[Weak solution of \hyperlink{link:EPH}{(EPH)}]  The pair $(S,p)$ with $p\in \W$, $ S-S_0\in\W$ and $S\in [0,1]$ a.e. in $Q$ is a weak solution of \hyperlink{link:EPH}{(EPH)} for $\tau>0$ if $p(0)=p_0$, $S(0)=S_0$ and it satisfies for all $\phi\in L^2(0,T;H^1_0(\Om))$ and $\xi\in L^2(Q)$,
 \begin{subequations}
\begin{numcases}{\hypertarget{link:Peph}{(\mathcal{P}_{\mathrm{EPH}})}}
\smallint_0^T\langle \p_t S, \phi\rangle= \smallint_0^T(k(S)[\del p-\bm{g}], \del \phi);\label{Eq:EPH1}\\
\smallint_0^T (\p_t (S+b(p)),\xi)=\smallint_0^T (\Phi_\tau(b(p),S+b(p)),\xi);\label{Eq:EPH2}
\end{numcases}
\end{subequations}
where $\Phi_\tau$ is defined in \eqref{Eq:DefPhi_eps}.\label{Def:WeakSolEPH}
\end{definition}

Observe that, according to Definition \ref{Def:WeakSolEPH}, $S$ has a trace on $\p\Om$ that does not change with time, i.e., it is fixed by $S_0$.

\section{Existence of solutions of $(\mathcal{P}w)$}\label{Sec:Existence}
The main existence result of this section is as follows:
\begin{theorem}
Assume \ref{ass:D}--\ref{ass:Ini}. Then $\hyperlink{link:Pw}{(\mathcal{P}w)}$ has a weak solution $(u,v)$ in terms of \Cref{def:WeakSol}. Moreover, $v\in L^\infty(0,T;H^1(\Om))$ and  $\p_t v\in L^\infty(0,T;L^2(\Om))$. 
\label{Theo:Existence}
\end{theorem}
To prove this, we apply Rothe's method \cite{kacur1986method}. Let the time $T$ be divided into $N$ time steps of width $\D t$ ($T=N\D t$) and for any $n\in\{1,....,N\}$ let $\p_t w$ be approximated by $(w_n - w_{n-1})/\D t$ for $w\in\{u,v\}$. Here $w_n$ stands for the value of the variable $w$ at time $t_n:=n\D t$. The time-discrete solution is defined as 
\begin{definition}[Time-discrete solution of $\hyperlink{link:Pw}{(\mathcal{P}w)}$]
For a given $n\in\{1,....,N\}$ and $(u_{n-1},v_{n-1})\in (L^2(\Om))^2$, the time-discrete solution of $\hyperlink{link:Pw}{(\mathcal{P}w)}$ at $t=t_n$ is a pair $(u_n,v_n)\in H^1_0(\Om)\times L^2(\Om)$ which satisfies for all $\phi\in H^1_0(\Om)$ and $\xi\in L^2(\Om)$,\label{def:TimeDiscreteSol}
\begin{subequations}\label{Eq:roth}
\begin{numcases}{\hypertarget{link:Pdt}{(\mathcal{P}^n_{\D t})}}
(u_n- u_{n-1},\phi)+\D t\, (\Df(u_{n-1},v_{n-1})\del u_n, \nabla \phi)\nonumber\\
\quad  =\D t\, (\vfl(u_{n-1},v_{n-1}), \nabla \phi)+\D t\, (\Psi_1(u_{n}, v_{n}), \phi);\label{Eq:roth1}\\
(v_n,\xi)=(v_{n-1},\xi)+ \D t\,(\Psi_2(u_{n},v_n),\xi) \label{Eq:roth2}.
\end{numcases}
\end{subequations}
\end{definition}
For the rest of the section the shorthand $\Df_n:=\Df(u_{n},v_{n})$, $\vfl_{n}:=\vfl(u_{n},v_{n})$ and $\Psi_{j,n}:=\Psi_j(u_n,v_n)$ will be used extensively for $n\in \{1,\dots,N\}$ and $j\in\{1,2\}$. 
We show first that the pair $(u_n,v_n)$  exists.
\begin{lemma}
Assume \ref{ass:D}--\ref{ass:Psi}. Then, there exists $\D t^*>0$ such that if $0<\D t<\D t^*$, a unique pair $(u_n,v_n)$ solving $\hyperlink{link:Pdt}{(\mathcal{P}^n_{\D t})}$ in terms of \Cref{def:TimeDiscreteSol} exists.
\end{lemma}
\begin{proof}
We define the operator $\mathcal{B}: (L^2(\Om))^2\to H^1_0(\Om)\times L^2(\Om)$ as follows: $\mathcal{B}(\tilde{u},\tilde{v})=(u^*,v^*)$, with $(u^*,v^*)$ solving
\begin{subequations}\label{Eq:LemmaBfpt}
\begin{align}
&(u^*,\phi)+\D t\, (\Df_{n-1}\nabla u^*-\vfl_{n-1}, \nabla \phi)=\D t\,(\Psi_1(\tilde{u},\tilde{v}), \phi)+(u_{n-1},\phi),\label{Eq:LemmaBfpt1} \\
&(v^*,\xi)=(v_{n-1},\xi)+ \D t\, (\Psi_2(\tilde{u},\tilde{v}),\xi),\label{Eq:LemmaBfpt2}
\end{align}
\end{subequations}
for any $\phi\in H^1_0(\Om)$ and $\xi\in L^2(\Om)$.  It follows from Lax-Milgram theorem \cite[Chapter 2]{zeidler1995applied} that $\mathcal{B}$ is well-defined. Next, we  apply the Banach fixed point theorem. For two pairs $(\tilde{u}_1,\tilde{v}_1)$ and $(\tilde{u}_2,\tilde{v}_2)$ in $H^1_0(\Om)\times L^2(\Om)$, let the outputs of $\mathcal{B}$ be $(u^*_1,v^*_1)$ and $(u^*_2,v^*_2)$. Subtracting the two versions of \eqref{Eq:LemmaBfpt1}, defining $e^*_u=u^*_1-u^*_2$, substituting $\phi=e^*_u$ and using Young's inequality we get
\begin{align}
&\norm{e^*_u}^2 + \D t \Df_m \norm{\del e^*_u}^2\leq (e^*_u,e^*_u)+ \D t\, (\Df_{n-1} \del e^*_u,\del e^*_u)= \D t\, (\Psi_1(\tilde{u}_1,\tilde{v}_1)-\Psi_1(\tilde{u}_2,\tilde{v}_2), e^*_u)\nonumber\\
&\leq \dfrac{\D t^2}{2}\norm{\Psi_1(\tilde{u}_1,\tilde{v}_1)-\Psi_1(\tilde{u}_2,\tilde{v}_2)}^2 + \dfrac{1}{2}\norm{e^*_u}^2\leq C\D t^2[\|\tilde{u}_1-\tilde{u}_2\|^2+ \|\tilde{v}_1-\tilde{v}_2\|^2] + \dfrac{1}{2}\norm{e^*_u}^2,\nonumber
\end{align}
for some $C>0$, the last inequality following from \ref{ass:Psi}. This implies,
\begin{align}
\frac{1}{2}\norm{e^*_u}^2+ \D t\, \Df_m \|\del e^*_u\|^2 <C\D t^2[\norm{\tilde{u}_1-\tilde{u}_2}^2 + \norm{\tilde{v}_1-\tilde{v}_2}^2].\label{Eq:LemmaUnexists}
\end{align}
Similarly defining $e^*_v=v^*_1-v^*_2$ we get from \eqref{Eq:LemmaBfpt2}
\begin{align}
\frac{1}{2}\norm{e^*_v}^2<C\D t^2[\norm{\tilde{u}_1-\tilde{u}_2}^2 + \norm{\tilde{v}_1-\tilde{v}_2}^2].\label{Eq:LemmaVnexists}
\end{align}
This clearly shows that $\mathcal{B}$ defines a contraction in $H^1_0(\Om)\times L^2(\Om)$ for $\D t$ small enough. More precisely, since $\|u\|_{\D t}:=\sqrt{\|u\|^2 + 2\D t\, \Df_m \|\del u\|^2}$ is an equivalent norm of $\|u\|_{H^1_0}$, we observe following  \eqref{Eq:LemmaUnexists} and \eqref{Eq:LemmaVnexists} that $\mathcal{B}(u,v)$ is contractive with respect to the norm $\sqrt{\|u\|_{\D t}^2 + \|v\|^2}$ for small $\D t$.
Hence, a fixed point $(u_n,v_n)$ of $\mathcal{B}$ exists in $H^1_0(\Om)\times L^2(\Om)$, i.e.,   $\mathcal{B}(u_n,v_n)=(u_n,v_n)$. This proves the lemma. We remark that the condition on $\D t$ is moderate, i.e. $(u_n,v_n)$  exists if $0<\D t\leq C$ where the constant $C>0$ depends neither on $n$ nor on $(u_{n-1},v_{n-1})$. 
\end{proof}
From now on we assume that $\D t$ is small enough which  guarantees the existence of solution pairs to the time discrete problems $\hyperlink{link:Pdt}{(\mathcal{P}^n_{\D t})}$.
Our goal will be to construct the solution $(u,v)$ from the time-discrete solutions. 
For this purpose, we introduce the following interpolation functions:
for $w\in \{u,v\}$ and $t\in (0,T]$ the piece-wise constant interpolations $\hat{w}_{\D t},\, \check{w}_{\D t}$ and the linear interpolation $\bar{w}_{\D t}$ are defined in $Q$ so that for $t\in(t_{n-1}, t_n]$ (recall that $t_n=n\D t$),
 \begin{align}
&\hat{w}_{\D t}(t)=w_n, \; \check{w}_{\D t}(t)=w_{n-1},\; \bar{w}_{\D t}(t)=w_{n-1} +\dfrac{t-t_{n-1}}{\D t} (w_n - w_{n-1}).\label{eq:Interpolation}
\end{align}

As a first step we show that $\bar{u}_{\D t}$ and $\bar{v}_{\D t}-v_0$ are bounded uniformly in $\W$ and then we would use embedding theorems to construct the weak solution.
\begin{lemma}
Let $(u_n,v_n)$ be the time-discrete solutions of $\hyperlink{link:Pdt}{(\mathcal{P}^n_{\D t})}$ in terms of \Cref{def:TimeDiscreteSol} for all $n\in\{1,\dots,N\}$ and let $\hat{u}_{\D t},\, \check{u}_{\D t},\, \bar{u}_{\D t}$ and $\hat{v}_{\D t},\, \check{v}_{\D t},\, \bar{v}_{\D t}$ be the interpolations defined in \eqref{eq:Interpolation}. Then, $\hat{u}_{\D t},\, \check{u}_{\D t},\, \bar{u}_{\D t} \in L^2(0,T;H^1_0(\Om))$ and $\p_t \bar{u}_{\D t}\in L^2(0,T;H^{-1}(\Om))$ and the corresponding norms are bounded uniformly with respect to ${\D t}$. Similarly, the bounds of $\hat{v}_{\D t},\, \check{v}_{\D t},\, \bar{v}_{\D t} \in L^\infty(0,T;H^1(\Om))$ and $\p_t \bar{v}_{\D t}\in L^\infty(0,T;L^2(\Om))$ are uniform.
\label{Lemma:ExistenceW}
\end{lemma}
\begin{proof}
\emph{(Step 1)} The fact that the functions $\hat{u}_{\D t},\, \check{u}_{\D t},\, \bar{u}_{\D t}$ belong to $L^2(0,T;H^1_0(\Om))$ is direct. We proceed by showing their uniform boundedness. Putting $\phi=u_n$ in \eqref{Eq:roth1} and $\xi=v_n$ in \eqref{Eq:roth2} we get,
\begin{align*}
&(\|u_n\|^2 -\|u_{n-1}\|^2 +\|u_n - u_{n-1}\|^2)+{\D t}\, \Df_m \|\nabla u_n\|^2\leq C{\D t} [1  +\|u_n\|^2 +\|v_n\|^2],\\
&(\|v_n\|^2 -\|v_{n-1}\|^2 +\|v_n - v_{n-1}\|^2)\leq C{\D t} [ 1 +\|u_n\|^2 +\|v_n\|^2],
\end{align*}
for some $C>0$.
Here, the identity $(a-b,b)=\frac{1}{2}(\norm{a}^2 - \norm{b}^2)+\frac{1}{2}\norm{a-b}^2$ has been invoked and $(\vfl_{n-1},\del u_n)\leq \frac{1}{2\Df_m}\|\vfl_{n-1}\|^2 + \frac{\Df_m}{2}\|\del u_n\|^2$ is used.
Combining both inequalities and summing the results up from $n=1$ to $P\leq N$ yields
\begin{align*}
&\|u_P\|^2 + \|v_P\|^2 +\sum_{k=1}^{P}(\|u_k-u_{k-1}\|^2 + \|v_k - v_{k-1}\|^2) + \Df_m \sum_{k=1}^{P}\|\del u_k\|^2 {\D t} \\
&\leq (\|u_0\|^2 + \|v_0\|^2) + 2 C P{\D t} + 2 C {\D t} \sum_{k=0}^{P}(\|u_k\|^2 + \|v_k\|^2).
\end{align*}
With $C_0=\|u_0\|^2 + \|v_0\|^2$ and a generic constant $C>0$, the discrete Gronwall's lemma is applied to get: 
\begin{equation}
 \|u_P\|^2 + \|v_P\|^2\leq \left( C_0 +C P\Delta t \right)\exp \{2 C P\Delta t \} \leq  ( C_0 + C T )\exp \{2CT \}.\label{Eq:ExistenceIneqLinfL2}
\end{equation}
Further, it gives two other important bounds both of which are used later, i.e.
\begin{align}
\sum_{k=1}^{N}\|\del u_k\|^2 {\D t}< C_1(T),\quad &\sum_{k=1}^{N}(\|u_k-u_{k-1}\|^2  + \|v_k - v_{k-1}\|^2) < C_2(T)\label{Eq:ExistenceImpIneq},
\end{align}
with $C_{1/2}(T)>0$ being independent of $N$. This directly gives the bounds of $\hat{u}_{\D t},\, \check{u}_{\D t},\, \bar{u}_{\D t}$ in $L^2(0,T;H^1_0(\Om))$ since, for example,
\begin{align*}
&\int_{0}^{T} \|\hat{u}_{\D t}\|_{H^1_0(\Om)}^2 dt=\sum_{k=1}^{N} \|\del u_k\|^2 {\D t} \leq C_1(T),
\end{align*}
with the rest of the bounds following accordingly.

\emph{(Step 2)} We need to show that $\hat{v}_{\D t},\, \check{v}_{\D t},\, \bar{v}_{\D t} \in L^\infty(0,T;H^1(\Om))$. So far we have from \eqref{Eq:ExistenceIneqLinfL2} that $\hat{v}_{\D t}, \check{v}_{\D t}, \bar{v}_{\D t} \in L^\infty(0,T;L^2(\Om))$. Since \Cref{def:TimeDiscreteSol}  does not explicitly involve any spatial derivatives of $v_n$, in order to prove its spatial regularity, we use directional derivatives: $D^h w=(w(\bm{x}+\hat{\bm{e}}h)-w(\bm{x}))/h$ for $\bm{x}\in \Om$ and an arbitrary unit vector $\hat{\bm{e}}\in \R^d$. Choose an open subset $V\subset \Om$ such that $\mathrm{dist}(V,\p\Om)>2h>0$.  Since $v_n$, $v_{n-1}\in L^2(\Om)$, we have from \eqref{Eq:roth2} that
$$
v_n=v_{n-1} +\D t\, \Psi_{2,n} \text{ a.e. in } \Om,\;\text{ or } D^h v_n = D^h v_{n-1} + \D t\, D^h \Psi_{2,n} \text{ a.e. in } V.
$$
Then multiplying both sides by $D^h v_n$ and integrating over $V$ we get
\begin{align*}
& \int_V (D^h v_n - D^h v_{n-1}) D^h v_n={\D t} \int_V D^h \Psi_{2,n} D^h v_n\leq {\D t} C [\int_V |D^h u_n|^2 +  \int_V |D^h v_n|^2],
\end{align*}
where $C>0$ does not depend on $n$ or $\D t$. After summing from $n=1$ to $P$, where $P\leq N$ is chosen arbitrarily, and using $\sum_{k=1}^{P}\smallint_V|D^h u_k|^2 {\D t}\leq \bar{C}_3 \sum_{k=1}^{P} \|\del u_k\|^2 {\D t} \leq  \bar{C}_3 C_1(T) =:C_3 $ for some $\bar{C}_3>0$ (see Theorem 3, Chapter 5.8 of \cite{evans1988partial}) we have
\begin{align*}
&\int_V |D^h v_P|^2 -  \int_V|D^h v_0|^2 + \sum_{k=1}^{P}  \int_V|D^h v_k - D^h v_{k-1}|^2 \leq C_3 + 2 C  \sum_{k=1}^{P}   \int_V|D^h v_k|^2 {\D t}.
\end{align*}
With the application of discrete Gronwall's lemma one obtains that $\smallint_V |D^h v_n|^2$ is bounded independent of $V$ and $h$. The smoothness of the boundary $\p\Om$ further implies that we can extend $V$ to $\Om$ and hence by applying  Theorem 3, Chapter 5.8 of \cite{evans1988partial} we get that $\|\del v_n\|$ is bounded. Consequently, $\hat{v}_{\D t}, \check{v}_{\D t}, \bar{v}_{\D t}\in L^\infty(0,T;H^1(\Om))$.

\emph{(Step 3)}
Finally, we prove the regularity of time derivatives of $\bar{u}_{\D t}$ and $\bar{v}_{\D t}$. Observe that, as $\hat{u}_{\D t},\hat{v}_{\D t}\in L^\infty(0,T;L^2(\Om))$, $\Psi_2(\hat{u}_{\D t},\hat{v}_{\D t})\in L^\infty(0,T;L^2(\Om))$ which gives from \eqref{Eq:roth2},
\begin{align*}
\|\p_t \bar{v}_{\D t}\|= \frac{1}{{\D t}}\|v_n-v_{n-1}\|=\|\Psi_{2,n}\|=\|\Psi_2(\hat{u}_{\D t},\hat{v}_{\D t})\|\leq C_4,
\end{align*}
for some $C_4>0$ and $t\in (t_{n-1},t_n]$.
For $\p_t \bar{u}_{\D t}$, from \eqref{Eq:roth1} one has
\begin{align*}
&\|\p_t \bar{u}_{\D t}\|_{H^{-1}(\Om)}=\sup_{\|\phi\|_{H^1_0(\Om)}=1} \langle\p_t \bar{u}_{\D t},\phi\rangle=  \sup_{\|\phi\|_{H^1_0(\Om)}=1} \left(\dfrac{u_n-u_{n-1}}{{\D t}},\phi\right)\\
&\leq \Df_M \|\del u_n\| + \|\vfl_{n-1}\| + {\|\phi\|}\|\Psi_{1,n}\|=\Df_M \|\del \hat{u}_{\D t}\| + \|F_M\| + C_\Om\|\Psi_1(\hat{u}_{\D t},\hat{v}_{\D t})\|.
\end{align*} 
Here, the Poincar\'e inequality $\|\phi\|\leq C_\Om\|\del \phi\|$ has been used. 
Since $\hat{u}_{\D t}\in L^2(0,T;H^1_0(\Om))$ and $\Psi_1(\hat{u}_{\D t},\hat{v}_{\D t}) \in L^2(Q)$, we have that $\p_t \bar{u}_{\D t}\in L^2(0,T;H^{-1}(\Om))$.
\end{proof}

\begin{proof}[\textbf{Proof of \Cref{Theo:Existence}}]
\Cref{Lemma:ExistenceW} shows that $\bar{u}_{\D t}$ and $(\bar{v}_{\D t}- v_0)$ are bounded in $\W$ uniformly with respect to ${\D t}$. Hence, there exists a sequence of time-steps  $\{{\D t}_p\}_{p\in\N}$ with $\lim {\D t}_p=0$ such that
\begin{align}
\bar{u}_{{\D t}_p}\rightharpoonup u \text{ and } (\bar{v}_{{\D t}_p}-v_0)\rightharpoonup (v -v_0) \text{ weakly in } \W.\label{Eq:wweak}
\end{align}
Due to the compact embedding of $\W$ in $L^2(Q)$ this further implies that,
\begin{align*}
\bar{w}_{{\D t}_p}\rightarrow w, \text{ strongly in } L^2(Q) \text{ for } w\in \{u,v\}.
\end{align*}
From \eqref{Eq:ExistenceImpIneq} the following inequalities are obtained,
\begin{align*}
&\int_{0}^{T}\|\hat{u}_{\D t}- \bar{u}_{\D t}\|^2 dt  =\sum_{k=1}^{N}\int_{t_{k-1}}^{t_k}\left\|\dfrac{t_k-t}{{\D t}}(u_k-u_{k-1})\right\|^2 dt=\dfrac{{\D t}}{3}  \sum_{k=1}^{N}\|u_k-u_{k-1}\|^2 \leq \dfrac{C_2}{3}{\D t},\\
&\int_{0}^{T}\|\check{u}_{\D t}- \hat{u}_{\D t}\|^2 dt  =\sum_{k=1}^{N}\int_{t_{k-1}}^{t_k}\|(u_k-u_{k-1})\|^2 dt={\D t} \sum_{k=1}^{N}\|u_k-u_{k-1}\|^2 \leq {C}_2{\D t}.
\end{align*}
This shows that  $\hat{u}_{{\D t}_p}\to u$ in $L^2(Q)$ as $\D t_p \to 0$ since $\|u- \hat{u}_{{\D t}_p}\|\leq \|u- \bar{u}_{{\D t}_p}\|+ \|\bar{u}_{{\D t}_p}-\hat{u}_{{\D t}_p}\|$. By repeating this argument, one shows that the same holds for $\check{u}_{{\D t}_p}$. Hence,
\begin{align}
\bar{u}_{{\D t}_p},\check{u}_{{\D t}_p},\hat{u}_{{\D t}_p}\rightarrow u, \text{ strongly in } L^2(Q).\label{Eq:ustrong}
\end{align}
In an identical way, for $v_{\D t}$ there exists a subsequence ${\D t}_q\to 0$ such that
\begin{align}
\bar{v}_{{\D t}_q},\check{v}_{{\D t}_q},\hat{v}_{{\D t}_q}\rightarrow v, \text{ strongly in } L^2(Q).\label{Eq:vstrong}
\end{align}
Finally, due to the bounds of $\p_t\bar{u}_{\D t}$ and $\p_t\bar{v}_{\D t}$ given in \Cref{Lemma:ExistenceW}, there exists a sequence ${\D t}_m\to 0$ such that,
\begin{align}
\partial_t \bar{w}_{{\D t}_m} \rightharpoonup \partial_t w, \text{ weakly in } L^2(0,T;H^{-1}(\Om)) \text{ for } w\in\{u,v\}. \label{Eq:timeweak}
\end{align} 
We claim that $(u,v)$ solves $\hyperlink{link:Pw}{(\mathcal{P}w)}$. Let ${\D t}_r \to 0$ be a sequence that satisfies the limits \eqref{Eq:wweak}--\eqref{Eq:timeweak}. From \eqref{Eq:roth} we have for $\phi\in L^2(0,T;H^1_0(\Om))$ and $\xi\in L^2(Q)$,
 \begin{align*}
 &\smallint^T_0\langle\partial_t \bar{u}_{{\D t}_r}, \phi\rangle+ \smallint^T_0(\Df(\check{u}_{{\D t}_r}, \check{v}_{{\D t}_r})\nabla \hat{u}_{{\D t}_r}, \nabla \phi)=\smallint^T_0(\Psi_1(\hat{u}_{{\D t}_r}, \hat{v}_{{\D t}_r}),\phi) + \smallint^T_0(\vfl( \check{u}_{{\D t}_r}, \check{v}_{{\D t}_r}),\del \phi),\\
 &\smallint^T_0\langle \partial_t \bar{v}_{{\D t}_r}, \xi\rangle=\smallint^T_0(\Psi_2(\hat{u}_{{\D t}_r}, \hat{v}_{{\D t}_r}), \xi).
 \end{align*}
 Since $\partial_t \bar{v}_{{\D t}_r}\rightharpoonup \partial_t v\in L^2(\Om)$, the second equation directly gives \eqref{Eq:weak2} in the limit.
From \eqref{Eq:timeweak}, $ \smallint_{0}^{T}\langle\partial_t \bar{u}_{{\D t}_r}, \phi\rangle \to \smallint_{0}^{T}\langle\partial_t u, \phi\rangle$ and \eqref{Eq:ustrong}--\eqref{Eq:vstrong} gives  $ \smallint_{0}^{T}(\Psi_1(\hat{u}_{{\D t}_r}, \hat{v}_{{\D t}_r}), \phi)\to \smallint_{0}^{T}(\Psi_1(u,v), \phi)$ and $\smallint^T_0(\vfl(\check{u}_{{\D t}_r}, \check{v}_{{\D t}_r}),\del\phi) \to \smallint^T_0(\vfl(u,v),\del\phi)$. Convergence of the second term $(\Df(\check{u}_{{\D t}_r}, \check{v}_{{\D t}_r})\nabla \hat{u}_{{\D t}_r}, \nabla \phi)$ remains to be shown. For this, we first observe that $\Df(\check{u}_{{\D t}_r}, \check{v}_{{\D t}_r})\del \hat{u}_{{\D t}_r}$ is bounded in $L^2(Q)$ uniformly with respect to $\D t$. This means that a $\bm{\zeta}\in L^2(0,T; L^2(\Om)^d)$ exists such that $\Df(\check{u}_{{\D t}_r}, \check{v}_{{\D t}_r})\del \hat{u}_{{\D t}_r} \rightharpoonup \bm{\zeta}$ weakly. To prove that $\bm{\zeta}=\Df(u,v)\del u$, we restrict the test function to $\phi\in  C^\infty_0(Q)$. Using the strong convergence of $\Df(\check{u}_{{\D t}_r}, \check{v}_{{\D t}_r})$ and the weak convergence of $\del \hat{u}_{{\D t}_r}$ one gets with $C_\phi=\|\phi\|_{C^1}$ that 

 \begin{align*}
 &\left|\smallint_{0}^{T} (\Df(\check{u}_{{\D t}_r}, \check{v}_{{\D t}_r})\del \hat{u}_{{\D t}_r}-\Df(u,v)\del u, \del \phi)\right|
\\
& \leq \left | \smallint_{0}^{T}\big((\Df(\check{u}_{{\D t}_r}, \check{v}_{{\D t}_r})-\Df(u,v))\del \hat{u}_{{\D t}_r}, \del \phi\big)\right | + \left |\smallint_{0}^{T}(\Df(u,v)\del (\hat{u}_{{\D t}_r}-u), \del \phi)\right |\\
&\leq C_\phi \smallint_{0}^{T} | (\Df(\check{u}_{{\D t}_r}, \check{v}_{{\D t}_r}) - \Df(u,v),\del \hat{u}_{{\D t}_r})| + \left | \smallint_{0}^{T}(\del (\hat{u}_{{\D t}_r}-u), \Df(u,v)\del \phi) \right |\\
&\leq C_2(T) C_\phi  \| \Df(\check{u}_{{\D t}_r}, \check{v}_{{\D t}_r}) - \Df(u,v)\|_{L^2(Q)}+ \left | \smallint_{0}^{T}(\del (\hat{u}_{{\D t}_r}-u), \Df(u,v) \del \phi) \right | \to 0,
 \end{align*}
 \noindent
 as ${\D t}_r\to 0$. Since the weak limit is unique, we have $\Df(\check{u}_{{\D t}_r}, \check{v}_{{\D t}_r})\del \hat{u}_{{\D t}_r} \rightharpoonup \bm{\zeta}=\Df(u,v)\del u$. This shows that $(u,v)$ is a weak solution of $\hyperlink{link:Pw}{(\mathcal{P}w)}$.
\end{proof}

\section{Boundedness and existence of solutions of (EPH)}\label{Sec:BoundsandLimits}
\subsection{$L^\infty$-bounds on $u$ and $v$}
Next, we investigate whether a solution of $\hyperlink{link:Pw}{(\mathcal{P}w)}$ satisfies the maximum principle or not. This is an interesting question primarily due to two reasons. Firstly, the maximum principle is used to prove the existence of solutions of $\hyperlink{link:Peph}{(\mathcal{P}_{\mathrm{EPH}})}$ in the case when $k(0)=0$. This is discussed in details later. Secondly, for pseudo-parabolic equations arising from the regularisation parameter $\t$, it is shown in \cite{VANDUIJN2018232,mitra2018wetting} that the maximum principle does not hold. Having similar structure to the systems mentioned above, one might wonder if a maximum principle holds for $\hyperlink{link:Pw}{(\mathcal{P}w)}$. As it turns out, a maximum principle does hold in this case for the variables $u$ and $v$ under certain conditions on the advection term.  This is preferred as a property of the EPH since hysteresis alone is known not to cause deviation from the maximum principle \cite{VANDUIJN2018232}. However, it is to be noted that the maximum principle does not necessarily generalise to other advective terms, as is expected in the case of pseudo-parabolic equations.

\begin{proposition}
Assume that $\vfl(u,v)=\vfl(u)$. For  $v_l:=\inf\{ v_0\}$ and $v_r:=\sup \{ v_0\}$ let there exist a pair $\{u_l,u_r\}$  ($u_l<u_r$) such that $u_0(\bm{x})\in [u_l,u_r]$ for almost all $\bm{x}\in \Om$ and 
$$
\Psi_1(u_r,v_l)\leq 0\leq \Psi_2(u_r,v_l),\quad  \Psi_2(u_l,v_r)\leq 0\leq \Psi_1(u_l,v_r).
$$
Then the weak solution $(u,v)$ of \Cref{def:WeakSol} satisfies $v_l\leq v \leq v_r$ and $u_l \leq u \leq u_r$ a.e. \label{Pros:Bound}
\end{proposition}
The rationals behind the assumptions used in the proposition are explained in \Cref{sec:ExistenceEPH} in the context of \hyperlink{link:EPH}{(EPH)}. 

\begin{proof}
We only show the proof that $u<u_r$ and $v>v_l$ a.e. and omit the other half of the proof since the arguments are identical. For any $t\in [0,T]$, let $\chi_{t}:[0,T]\to \{0,1\}$ denote the characteristic function of $[0,t]$, i.e. $\chi_{t}=1$ in $[0,t]$ and $\chi_{t}=0$ everywhere else. Taking $\phi= [u-u_r]_+\, \chi_{t}\in L^2(0,T;H^1_0(\Om))$ in \eqref{Eq:weak1} and recalling that $\smallint_0^T (\del \zeta, \del [\zeta]_+) =\smallint_0^T\|\del [\zeta]_+\|^2$ and $\smallint_0^T \langle[\zeta]_+,\p_t \zeta\rangle =\frac{1}{2} \| [\zeta]_+\|^2\big|_{t=0}^T$ for all $\zeta\in \W$ one has:
\begin{align*}
&\dfrac{1}{2} \|[u(t)-u_r]_+\|^2 + \Df_m \smallint_{0}^{t}  \|\del [u-u_r]_+\|^2 \leq \smallint_{0}^{t}\langle \p_t u, [u-u_r]_+\rangle +  \smallint_{0}^{t}(\Df(u,v)\del u ,\del [u-u_r]_+)\\
&= \smallint_{0}^{t}(\Psi_1(u,v),[u-u_r]_+) + \smallint_{0}^{t}(\vfl(u),\del [u-u_r]_+)\\
&\leq  \smallint_{0}^{t}\left(\Psi_1(u,v)-\Psi_1(u_r,v_l),[u-u_r]_+\right)+ \smallint_{0}^{t}(\vfl(u)-\vfl(u_r),\del [u-u_r]_+)\\
&= \smallint_{0}^{t}  \left(\Psi_1(u,v)-\Psi_1(u_r,v)+\Psi_1(u_r,v)-\Psi_1(u_r,v_l),[u-u_r]_+ \right) + \smallint_{0}^{t}(\p_u\vfl\, [u-u_r]_+,\del [u-u_r]_+)\\
&\leq   \Psi_u \smallint_{0}^{t}\|[u-u_r]_+\|^2 + \smallint_{0}^{t}\left(-\tfrac{\Psi_1(u_r,v)-\Psi_1(u_r,v_l)}{v-v_l}(v_l-v),[u-u_r]_+\right)+ C\smallint_{0}^{t}\|[u-u_r]_+\|\|\del [u-u_r]_+\|\\
&\leq \Psi_u \smallint_{0}^{t}\|[u-u_r]_+\|^2 + \Psi_v\smallint_{0}^{t}\left([v_l-v]_+,[u-u_r]_+\right)+ \frac{C^2}{2\Df_m}\smallint_{0}^{t}\|[u-u_r]_+\|^2 + \frac{\Df_m}{2}\smallint_{0}^{t}\|\del [u-u_r]_+\|^2\\
&\leq C_m\smallint_{0}^{t}[\|[u-u_r]_+\|^2 +\| [v_l-v]_+\|^2]+ \frac{\Df_m}{2}\smallint_{0}^{t}\|\del [u-u_r]_+\|^2. 
\end{align*}
Here $C,\,C_m>0$ are constants. The inequality $0\leq-\tfrac{\Psi_1(u_r,v)-\Psi_1(u_r,v_l)}{v-v_l}\leq \Psi_v$ follows from \ref{ass:Psi}.
 Moreover, $\vfl\in C^1(\R)$ and $\int_\Om \vfl(u_r)\cdot\del[u-u_r]_+=0$ are used. 
 
Similarly, using the test function $\xi= [v_l-v]_+ \chi_{t} $ in \eqref{Eq:weak2} yields
\begin{align*}
&\dfrac{1}{2} \norm{[v_l-v(t)]_+}^2= \smallint_{0}^{t}(-\Psi_2(u,v),[v_l-v]_+) \leq  \smallint_{0}^{t}(\Psi_2(u_r,v_l)-\Psi_2(u,v),[v_l-v]_+)\\
& \leq C_m\smallint_{0}^{t}[\|[u-u_r]_+\|^2 +\| [v_l-v]_+\|^2].
\end{align*}
Finally adding them yields for all $t\in [0,T]$ that
\begin{equation}
\norm{[u(t)-u_r]_+}^2 + \norm{[v_l-v(t)]_+}^2 \leq 4 C_m \smallint_{0}^{t} [\norm{[u-u_r]_+}^2 + \norm{[v_l-v]_+}^2].
\end{equation}
Since $\norm{[v_l-v_0]_+}=0$ and $\norm{[u_0-u_r]_+}^2=0$, we conclude from Gronwall's lemma that $ \norm{[u(t)-u_r]_+}=0$ and $\norm{[v_l-v(t)]_+}=0$ for all $t>0$. This proves the proposition.
\end{proof}

\subsection{Existence of solutions of (EPH)}\label{sec:ExistenceEPH}
The main result of this section is the existence of a solution to \hyperlink{link:Peph}{$(\mathcal{P}_{\mathrm{EPH}})$}. This is obtained first for the case when $k(0)>0$, and then for $k(0)=0$ in the absence of convective terms. 
\begin{theorem}
Assume \ref{prop:Pc}--\ref{prop:Rc}. 
\begin{itemize}
\item[(a)] If $k(0)>0$ then a solution of $\hyperlink{link:Peph}{(\mathcal{P}_{\mathrm{EPH}})}$ in terms of \Cref{Def:WeakSolEPH} exists.
\item[(b)] If $|\bm{g}|=0$, then a solution $(S,p)$ of $\hyperlink{link:Peph}{(\mathcal{P}_{\mathrm{EPH}})}$ exists for $k(0)=0$. Moreover, there exists saturations $0<S_l<S_r<1$ such that $S(\bm{x},t)\in [S_l,S_r]$ and $p(\bm{x},t)\in [\Pim(S_r),\Pdr(S_l)]$ for almost all $(\bm{x},t)\in Q$.
\end{itemize}\label{Theo:ExistenceEPH}
\end{theorem}

To begin with, we observe that \ref{prop:b} implies $b'(p)\to 0$ for $p\to \pm \infty$ which might cause the model to become degenerate. 
Therefore, we consider a non-degenerate system that approximates $\hyperlink{link:Peph}{(\mathcal{P}_{\mathrm{EPH}})}$ first. Let $\rim(v)$ and $\rdr(v)$, defined in \eqref{def:Rimdr}, be extended to $\R$ such that $\rim(v)=\rdr(v)=U_M$ for $v\leq V_m$, and $\rim(v)=\rdr(v)=U_m$ for $v\geq V_M$. For $\Phi_\tau$ defined in \eqref{Eq:DefPhi_eps} this implies that $\|\Phi_\tau(u,v)\|^2 < C[1+\|u\|^2]$, since $\rim,\;\rdr$ are bounded.

For some $\d>0$ small enough, let $\Pim(1)<p^\d_l<p^\d_r<\infty$ be such that $b'(p^\d_l)=b'(p^\d_r)=\d$. Without loss of generality, assume  that $b'(p)>\d$ for $p^\d_l<p<p^\d_r$. Define $b_\d\in C^1(\R)$ to be a regularised version of $b(p)$ such that (see \Cref{fig:TheoremEPH} (left))
\begin{enumerate}[label=(P$\d$)]
\item $b_\d'(p)\geq \d>0$ for all $p\in \R$ with $b_\d'(p)=\d$ for $\{p\leq p^\d_l\}\cup\{p\geq p^\d_r\}$ and $b_\d(p)=b(p)$ for $p^\d_l<p<p^\d_r$. Clearly, $|b_\d(p)-b(p)|<C\d[|p|+1]$ for some constant $C>0$.\label{prop:Pdelta}
\end{enumerate}
Consequently, as $p_0$ satisfies  \ref{prop:Ini}, $p^\d_l<p_0<p^\d_r$ for $\d$ small enough, making $b_\d(p_0)=b(p_0)$.

\begin{figure}[h!]
\begin{subfigure}{.5\textwidth}
\begin{tikzpicture}[xscale=.35,yscale=3,domain=0:1,samples=100]

\draw[->] (0,-.2)--(0,1.2) node[right, scale=1.5] {$u$};
\draw[->] (-10,-0.1)--(10,-0.1) node[above,scale=1.5] {$p$};

\draw[dashed] (10,1)--(0,1) node[scale=1,left] {$U_M$};
\draw[dashed] (-10,0)--(0,0);
\node[scale=1,above] at (-9,0) {$U_m$};

\draw[ultra thick, green] (-10,0.0024814)--(-9.6,0.0026908)--(-9.2,0.0029278)--(-8.8,0.0031974)--(-8.4,0.0035059)--(-8,0.0038611)--(-7.6,0.0042729)--(-7.2,0.0047539)--(-6.8,0.0053204)--(-6.4,0.005994)--(-6,0.006803)--(-5.6,0.0077862)--(-5.2,0.0089968)--(-4.8,0.01051)--(-4.4,0.012434)--(-4,0.014929)--(-3.6,0.018241)--(-3.2,0.02276)--(-2.8,0.029129)--(-2.4,0.038462)--(-2,0.052786)--(-1.6,0.076001)--(-1.2,0.11589)--(-0.8,0.18765)--(-0.4,0.3143)--(0,0.5)--(0.4,0.6857)--(0.8,0.81235)--(1.2,0.88411)--(1.6,0.924)--(2,0.94721)--(2.4,0.96154)--(2.8,0.97087)--(3.2,0.97724)--(3.6,0.98176)--(4,0.98507)--(4.4,0.98757)--(4.8,0.98949)--(5.2,0.991)--(5.6,0.99221)--(6,0.9932)--(6.4,0.99401)--(6.8,0.99468)--(7.2,0.99525)--(7.6,0.99573)--(8,0.99614)--(8.4,0.99649)--(8.8,0.9968)--(9.2,0.99707)--(9.6,0.99731)--(10,0.99752);

\draw[orange, ultra thick, dashdotted] (-6,-.2)--(-1.5,0.083)--(-1.2,0.11589)--(-0.8,0.18765)--(-0.4,0.3143)--(0,0.5)--(0.4,0.6857)--(0.8,0.81235)--(1.2,0.88411)--(1.5,0.916)--(6,1.2);

\draw[dashed] (1.5,0.916)--(1.5,-.1) node[scale=1,below] {$p^\d_r$};
\draw[dashed] (-1.5,0.083)--(-1.5,-.1) node[scale=1,below] {$p^\d_l$};
\draw[ultra thick, dotted] (3.9714,1.0744)--(4.5438,1.0287)--(4.8658,0.9601)--(4.93,0.916);
\draw[thick] (1.5,0.916)--(7,.916);
\node[scale=0.8,below] at (8,.8) {$\tan^{-1}(\d)$};
\draw[->] (8,.8)--(4.5438,1.0287);

\node[draw=black,thick,rounded corners=2pt,below left=2mm, scale=1] at (-2,.8){%
\begin{tabular}{@{}r@{ }l@{}}
 \raisebox{2pt}{\tikz{\draw[ultra thick, green, domain=0.05:0.06] (0,0) -- (5mm,0);}}&$b(p)$\\
  \raisebox{2pt}{\tikz{\draw[dashed,ultra thick, orange, domain=0.05:0.06] (0,0) -- (5mm,0);}}&$b_\d(p)$
\end{tabular}};
\end{tikzpicture}
\end{subfigure}
\begin{subfigure}{.5\textwidth}
\begin{tikzpicture}[xscale=5.5,yscale=3.5,domain=0:1,samples=100]

\draw[->] (0,0)--(0,1.1) node[right, scale=1.5] {$u$};
\draw[->] (0,0)--(1.15,0) node[above,scale=1.5] {$v$};
\node[scale=1,below left] at (0,0) {$U_m$};

\draw[ultra thick, red] (1,0)--(1,1.4089e-16)--(1,9.417e-13)--(1,7.0296e-11)--(1,1.4999e-09)--(1,1.6107e-08)--(1,1.1204e-07)--(1,5.7752e-07)--(1,2.3906e-06)--(1,8.3691e-06)--(0.99999,2.5672e-05)--(0.99999,7.0762e-05)--(0.99996,0.00017856)--(0.99992,0.00041836)--(0.99982,0.00092013)--(0.99962,0.0019161)--(0.99924,0.0038041)--(0.99855,0.0072397)--(0.99735,0.013265)--(0.9953,0.023476)--(0.99195,0.040225)--(0.98663,0.066826)--(0.97846,0.10768)--(0.96636,0.16819)--(0.94915,0.25423)--(0.9258,0.37101)--(0.88956,0.49029)--(0.83589,0.58598)--(0.7684,0.66167)--(0.69258,0.72103)--(0.61469,0.76734)--(0.5403,0.80341)--(0.47335,0.83151)--(0.41584,0.85343)--(0.36813,0.87058)--(0.32956,0.88404)--(0.2989,0.89464)--(0.27481,0.90302)--(0.25601,0.90967)--(0.24139,0.91498)--(0.23002,0.91922)--(0.22118,0.92264)--(0.21429,0.9254)--(0.20891,0.92763)--(0.20468,0.92945)--(0.20136,0.93093)--(0.19873,0.93215)--(0.19664,0.93315)--(0.19497,0.93397)--(0.19364,0.93466)--(0.19257,0.93523)--(0,.935);

\draw[blue, ultra thick, dashdotted] (1,0)--(1,1.4089e-16)--(0.99999,9.417e-13)--(0.99996,7.0296e-11)--(0.9999,1.4999e-09)--(0.99978,1.6107e-08)--(0.99953,1.1204e-07)--(0.99903,5.7752e-07)--(0.99811,2.3906e-06)--(0.99644,8.3691e-06)--(0.99356,2.5672e-05)--(0.98874,7.0762e-05)--(0.98091,0.00017856)--(0.96864,0.00041836)--(0.95003,0.00092013)--(0.92288,0.0019161)--(0.88495,0.0038041)--(0.83451,0.0072397)--(0.7712,0.013265)--(0.69674,0.023476)--(0.61517,0.040225)--(0.53236,0.066826)--(0.45473,0.10768)--(0.38795,0.16819)--(0.33605,0.25423)--(0.30114,0.37101)--(0.27729,0.49029)--(0.25815,0.58598)--(0.24301,0.66167)--(0.23114,0.72103)--(0.22188,0.76734)--(0.21466,0.80341)--(0.20904,0.83151)--(0.20466,0.85343)--(0.20123,0.87058)--(0.19854,0.88404)--(0.19642,0.89464)--(0.19474,0.90302)--(0.19341,0.90967)--(0.19235,0.91498)--(0.1915,0.91922)--(0.19082,0.92264)--(0.19027,0.9254)--(0.18982,0.92763)--(0.18946,0.92945)--(0.18916,0.93093)--(0.18891,0.93215)--(0.18871,0.93315)--(0.18855,0.93397)--(0.18841,0.93466)--(0.1883,0.93523);

\draw[dashed] (0.1883,0.93523)--(.1883,0) node[scale=1,below] {$V_m$};
\node[scale=1,below] at (1,0) {$V_M$};
\draw[dashed] (.45,0.93523)--(0,0.93523) node[scale=1,above left] {$U_M$};
\draw[thick] (0.05,0.875)--(.3,1) node[scale=0.8,above] {$S=0$};
\draw[thick] (0.05,0.765)--(.52,1) node[scale=0.8,above right] {$S=S_l$};
\draw[thick] (0.05,0.625)--(.7,.95) node[scale=0.8,right] {$S=\epsilon$};
\draw[thick] (0.45,0)--(.75,.15);
\node[scale=0.8,above] at (.8,.15) {$S=1-\epsilon$};
\draw[thick] (0.64,0)--(.84,.1);
\node[scale=0.8,right] at (.82,.1) {$S=S_r$};

\draw[ultra thick,dotted] (0.3,0.37)--(.3,.89);
\draw[thick,dashed] (0.3,0.37)--(0.3,0) node[scale=1,below] {$v_l$};
\draw[thick,dashed] (.3,.89)--(0,0.89) node[scale=1,left] {$u_r$};

\draw[ultra thick,dotted] (0.6768,0.025)--(0.6768,.7335);
\draw[thick,dashed] (0.6768,0.025)--(0.6768,0) node[scale=1,below] {$v_r$};
\draw[thick,dashed] (0.6768,0.025)--(0,0.025) node[scale=1,left] {$u_l$};

\draw[fill,pink!60] (0.54839,0.73776)--(0.47755,0.72145)--(0.44052,0.6958)--(0.40349,0.66783)--(0.39551,0.66618)--(0.35403,0.59621)--(0.30841,0.48834)--(0.30012,0.43586)--(0.33537,0.32507)--(0.38514,0.26093)--(0.42454,0.22012)--(0.48882,0.18513)--(0.54482,0.17638)--(0.55104,0.17638)--(0.60635,0.17599)--(0.64338,0.21795)--(0.66592,0.29254)--(0.67397,0.37179)--(0.64982,0.47436)--(0.62406,0.60023)--(0.58381,0.68881)--(0.54034,0.73776);

\node[scale=1] at (0.5133,0.4418) {$(v_0,u_0)$};

\node[draw=black,thick,rounded corners=2pt,below left=2mm, scale=1] at (1.2,.9){%
\begin{tabular}{@{}r@{ }l@{}}
 \raisebox{2pt}{\tikz{\draw[ultra thick, red, domain=0.05:0.06] (0,0) -- (5mm,0);}}&$\rdr$\\
  \raisebox{2pt}{\tikz{\draw[dashed,ultra thick, blue, domain=0.05:0.06] (0,0) -- (5mm,0);}}&$\rim$
\end{tabular}};
\end{tikzpicture}
\end{subfigure}
\caption{(left) Schematic for $b(p)$ and $b_\d(p)$ used in the proof of \Cref{Theo:ExistenceEPH}(a). The values $p^\d_l$, $p^\d_r$, $U_m$ and $U_M$ are marked. (right) The bounds of $u$, $v$ and $S$ from  \Cref{Theo:ExistenceEPH} (b) shown in the $v$-$u$ plain. The black parallel lines are representing $v-u=S=\textit{constant}$. The values $v_l$, $v_r$, $u_l$, $u_r$, $S_l$ and $S_r$ are shown, as well as the set $(v_0,u_0)$.}\label{fig:TheoremEPH}
\end{figure}
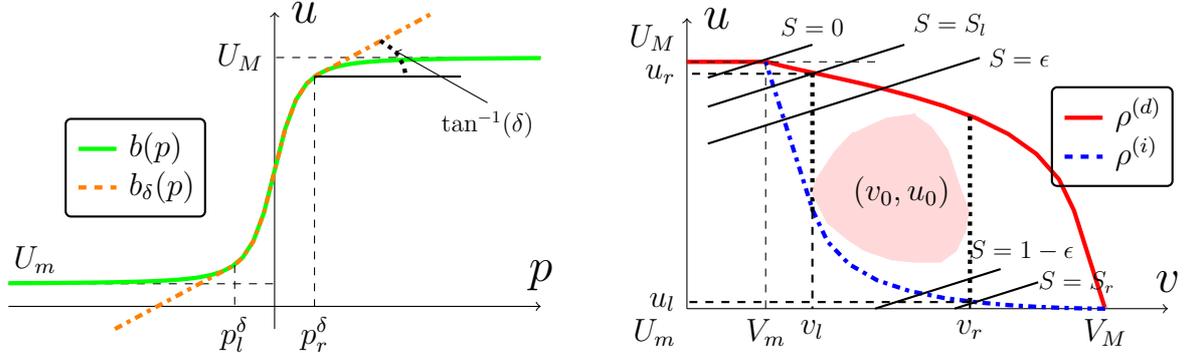

We look for solutions $(p_\d,u_\d,v_\d)$, with $p_\d,\,u_\d\in \W$ and $v_\d -S_0 -b(p_0)\in \W$, such that for any $\phi\in L^2(0,T;H^1_0(\Om))$ and $\xi\in L^2(Q)$ the following is satisfied
\begin{subequations}\label{Eq:bdelta}
\begin{numcases}{\hypertarget{link:Pd}{(\mathcal{P}_\d)}}
\smallint_0^T\langle \p_t u_\d, \phi\rangle + \smallint_0^T(k(v_{\d}-u_\d)[\del p_{\d}-\bm{g}], \del \phi)=\smallint_0^T( \Phi_\tau(u_\d,v_\d),\phi);\label{Eq:bdelta1}\\
(\p_t v_\d,\xi)=(\Phi_\tau(u_\d,v_\d),\xi),\quad  u_\d= b_\d(p_\d)\quad  \text{ in } \bar{Q};\label{Eq:bdelta2}
\end{numcases}
\end{subequations}
with $u_\d(0)=b(p_0)$ and $v_\d(0)=S_0+ b(p_0)$. 

The existence of such a triplet $(p_\d,u_\d,v_\d)$ follows from \Cref{Theo:Existence} by setting $\vfl$, $\Psi_1$, $\Psi_2$ as in \eqref{Eq:DefkPsi12} and $\Df(u,v):=\frac{k(v-u)}{b_\d'(b_\d^{-1}(u))}$. For $k(0)>0$ it follows directly that all the assumptions \ref{ass:D}--\ref{ass:Ini} of \Cref{Theo:Existence} are satisfied. Hence $(p_\d,u_\d,v_\d)$ exists. We show uniform bounds of $p_\d,\;u_\d$ and $v_\d$ with respect to $\d$ for this case.
\begin{lemma}
Assume \ref{prop:Pc}--\ref{prop:Rc}, \ref{prop:Pdelta}, $k(0)>0$ and $\d$ small so that $b_\d(p_0)=b(p_0)$. Let $(p_\d,u_\d,v_\d)$ be a solution of the problem $\hyperlink{link:Pd}{(\mathcal{P}_\d)}$, i.e. \eqref{Eq:bdelta}. 
Then, $p_\d$ and $u_\d$ are uniformly bounded in $L^2(0,T;H^1_0(\Om))$ and $\W$ respectively, whereas, $v_\d-S_0-b(p_0)$ is uniformly bounded in $\W$.\label{lemma:deltaBoundsinW}
\end{lemma}
\begin{proof}
To show this, we first use the test function $\phi=u_\d$ in \eqref{Eq:bdelta1} which gives
\begin{align*}
&\|u_\d(T)\|^2- \|b(p_0)\|^2+ \frac{k(0)}{b_M} \smallint_0^T \|\del u_\d\|^2\\
&\leq \smallint_0^T\langle \p_t u_\d, u_\d\rangle + \smallint_0^T(k(v_{\d}-u_{\d})\del p_{\d}, b_\d'(p_\d)\del p_\d)\\
&\leq \smallint_0^T(k(v_{\d}-u_{\d})\,\bm{g},\del u_\d) + \smallint_0^T(\Phi_\tau,u_\d)\leq  \frac{b_M}{2k(0)}\|k(1)\,\bm{g}\|^2 T + \frac{k(0)}{2 b_M}\smallint_0^T\|\del u_\d\|^2 + C\smallint_0^T [1+ \|u_\d\|^2].
\end{align*}
Here $b_M:=\max_{\vr\in \R}b'(\vr)= \max_{\vr\in \R}b_{\d}'(\vr)$. Using Gronwall's lemma, we directly get the boundedness of $u_\d$ in $L^2(0,T;H^1_0(\Om))$. Using the test function $\phi=p_\d$ in \eqref{Eq:bdelta1} gives
\begin{align*}
&\left\|\smallint_0^{p_\d(T)} \vr b_\d'(\vr)d\vr\right\|_{1}-\left\|\smallint_0^{p_0} \vr b_\d'(\vr)d\vr\right\|_{1}+ k(0) \|\del p_\d\|^2\\
&\leq \smallint_0^T\langle b'(p_\d)\p_t p_\d, p_\d\rangle + \smallint_0^T(k(v_{\d}-u_{\d})\del p_{\d}, \del p_\d)= \smallint_0^T\langle \p_t u_\d, p_\d\rangle + \smallint_0^T(k(v_{\d}-u_{\d})\del p_{\d}, \del p_\d)\\
&= \smallint_0^T(k(v_{\d}-u_{\d})\,\bm{g},\del p_\d) + \smallint_0^T(\Phi_\tau,p_\d)\leq \smallint_0^T(k(v_{\d}-u_{\d})\,\bm{g},\del p_\d) + \smallint_0^T\|\Phi_\tau\|\|p_\d\|\\
&\leq  \frac{\|k(1)\,\bm{g}\|^2}{k(0)} T + \frac{k(0)}{4}\smallint_0^T\|\del p_\d\|^2 + \dfrac{C^2_\Om}{k(0)}\smallint_0^T\|\Phi_\tau\|^2 + \frac{k(0)}{4}\smallint_0^T \|\del p_\d\|^2.
\end{align*}
We have used Poincar$\acute{\text{e}}$ inequality $\|p_\d\|\leq C_\Om \|\del p_\d\|,$ and the fact that $\smallint_0^{p} \vr b_\d'(\vr)d\vr\geq 0$ for all $p\in \R$   here. This  gives that $p_\d$ is bounded uniformly in $L^2(0,T;H^1_0(\Om))$. Finally, by following the steps of  \Cref{Lemma:ExistenceW} (Step 2), one concludes that $v_\d\in L^\infty(0,T;H^1(\Om))$ while $\p_t v_\d\in  L^\infty(0,T;L^2(\Om))$, the bounds being uniform in both cases.
Moreover, following the steps of (Step 3) of \Cref{Lemma:ExistenceW}
$$
 \|\p_t u_\d\|_{H^{-1}}\leq k(1)\|\bm{g}\| + k(1) \|\del p_\d\| + C_\Om\|\Phi_\tau\|,
$$
which shows that $\p_t u_\d\in L^2(0,T; H^{-1}(\Om))$ is bounded uniformly, thus, concluding the proof.
\end{proof}
\begin{proof}[\textbf{Proof of \Cref{Theo:ExistenceEPH}(a)}]
Define $S_\d:=v_\d-u_\d$. From  \Cref{lemma:deltaBoundsinW} it follows that there exists a sequence $\{\d_r\}_{r\in\N}$ with $\lim \d_r=0$ such that ($\to$ implies strong and $\rightharpoonup$ implies weak convergence)
\begin{align*}
&u_{\d_r} \to u,\; v_{\d_r}\to v,\; S_{\d_r}\to S=v-u \text{ in } L^2(Q);\\
&p_{\d_r} \rightharpoonup p  \text{ in } L^2(0,T;H^1_0(\Om)),\; \p_t u_{\d_r} \rightharpoonup \p_t u \text{ in } L^2(0,T;H^{-1}(\Om)) \text{ and } \p_t v_{\d_r} \rightharpoonup \p_t v \text{ in } L^2(Q).
\end{align*}
Hence, following the proof of \Cref{Theo:Existence} we conclude  that $u,v,p$ and $S$ satisfy
$$
\smallint_0^T\langle \p_t S, \phi\rangle = \smallint_0^T(k(S)[\del p-\bm{g}], \del \phi)\quad \text{ and } \quad \smallint_0^T(\p_t v,\xi)= \smallint_0^T (\Phi_\tau(u,v),\xi),
$$
for all $\phi\in L^2(0,T;H^1_0(\Om))$ and $\xi\in L^2(Q)$. For completeness, we still need to show that $u=b(p)$ and $S\in [0,1]$ a.e.
To show the former, observe that $b(p_{\d_r})\to u$ in $L^2(Q)$, since,
\begin{equation}
\|u-b(p_{\d_r})\|\leq \|u-u_{\d_r}\|+\|b_{\d_r}(p_{\d_r})-b(p_{\d_r})\|\to 0. \label{Eq:UtoBpd}
\end{equation}
The first term on the right vanishes as $u_\d \to u$ strongly. For the second we use \ref{prop:Pdelta}, giving $\|b_{\d_r}(p_{\d_r})-b(p_{\d_r})\|\leq C\d_r [1+ \|p_{\d_r}\|]$ for some $C>0$, which approaches 0 as $\d_r\to 0$.
Now, from \ref{prop:b} one gets for $b_M=\max_{\vr\in \R}b'(\vr)<\infty$,
\begin{align*}
&\|b(p)- b(p_{\d_r})\|^2\leq b_M (b(p)- b(p_{\d_r}),p- p_{\d_r})\\
&=b_M(b(p)- u,p- p_{\d_r} )+ b_M (u-b(p_{\d_r}),p- p_{\d_r} ).
\end{align*}
As $\d_r\to 0$, the first term in the right vanishes due to the weak convergence of $p_{\d_r}$ and the second term vanishes due to \eqref{Eq:UtoBpd}, thus proving $u=b(p)$.

Finally,  $u=b(p)\in [U_m,U_M]$ a.e. implies that $v\in [V_m,V_M]$ a.e. To see this, we use the test function $\xi=[V_m-v]_+$ in \eqref{Eq:bdelta2} to get
$$
\tfrac{1}{2}\|[V_m-v]_+(T)\|^2=\tfrac{1}{2}\smallint_0^T \p_t \|[V_m-v]_+\|^2=(-\Phi_\tau(u,v),[V_m-V]_+)=\tfrac{1}{\t}(u-\rim(v),[V_m-V]_+)\leq 0.
$$
Here, as $v<V_m$ and $u<U_M$, $\Phi_\tau=\frac{\rim(v)-u}{\t}\geq \frac{\rim(V_m)-U_M}{\t}= 0$, see also \Cref{fig:TheoremEPH} (right). This proves that 
$\|[V_m-v]_+\|=0$ meaning $v\geq V_m$ a.e. implying that $S=v-u\geq V_m-U_M=0$. Similarly one shows that $S\leq 1$ which concludes our proof.
\end{proof}

\begin{proof}[\textbf{Proof of \Cref{Theo:ExistenceEPH}(b)}] Since $k(0)=0$, we approximate $k$ by $k_\mu:\R\to\R^+$, $\mu>0$, satisfying $k_\mu(S)= k(\mu)>0$ for $S<\mu$ and $k_\mu(S)=k(S)$ otherwise. Then the solution $(S_\mu,p_\mu)$ of $\hyperlink{link:Peph}{(\mathcal{P}_{\mathrm{EPH}})}$ exists with $k_\mu$ replacing $k$. Define $u_\mu=b(p_\mu)$ and $v_\mu=S_\mu+b(p_\mu)$, as before. The pair $\{v_l, v_r\}$ is defined as in \Cref{Pros:Bound}, i.e. $v_l=\inf\{S_0 + b(p_0)\},\;v_r=\sup\{S_0 + b(p_0)\}$. Note that, $v_l,v_r\in (V_m,V_M)$ due to \ref{prop:Ini}, see \Cref{fig:TheoremEPH} (right). Define $u_l,u_r\in (U_m,U_M)$ as
$$
u_r=\rdr(v_l),\quad u_l=\rim(v_r) \quad  \text{ giving }\quad \Phi_\tau(v_l,u_r)=\Phi_\tau(v_r,u_l)=0,
$$
see \Cref{fig:TheoremEPH} (right) for clarification. Moreover, since $\Pim(S_0)\leq p_0 \leq \Pdr(S_0)$, it implies that $u_0=b(p_0)\in [u_l,u_r]$. Thus, recalling $|\bm{g}|=0$, we have from  \Cref{Pros:Bound} that $u_l\leq u_\mu\leq u_r$ and $v_l\leq v_\mu\leq v_r$ a.e. Consequently, $S_\mu=v_\mu-u_\mu\geq v_l-u_r>0$. Defining 
\begin{equation}
S_l:=v_l-u_r \text{ and } S_r:=v_r-u_l \text{ one has } 0<S_l\leq S_\mu\leq S_r<1 \text{ a.e. in } Q.
\end{equation}
By taking $\mu<S_l$, one actually gets $k_\mu(S_\mu)\geq k(S_l)>0$ a.e. in $Q$. Thus, passing $\mu\to 0$, we obtain the solution.
\end{proof}

\section{Behaviour when $\t\to 0$}\label{Sec:EpsToZero}
In this section we address one other important question: what is the behavior of $(S,p)$ for the limit $\t\to 0$ which corresponds to the unregularised EPH model. Ideally $\Pim(S)\leq p\leq \Pdr(S)$ (or equivalently  $\rim(v)\leq u\leq \rdr(v)$) should be satisfied in the pure hysteresis case, i.e., when $\t\to 0$. One might also wonder whether a limiting solution exists in this case or not. The problem is addressed for the play-type hysteresis model \eqref{Eq:playtype} in \cite[Theorem 3.2]{schweizer2007averaging} for a reduced case (linear, no advection etc.), however, with stochastic variation of coefficients included. For the nonlinear case, no results are available to our knowledge even for the play-type model. In this section, we take a step towards understanding the limiting case $\t\to 0$. 
\begin{proposition} Let the assumptions of \Cref{Theo:ExistenceEPH}  hold. For a given $\t>0$, let $(S_\t,p_\t)$ denote a weak solution of $\hyperlink{link:Peph}{(\mathcal{P}_{\mathrm{EPH}})}$ in terms of \Cref{Def:WeakSolEPH}. Further let $u_\t:=b(p_\t)$ and $v_\t:=S_\t+ b(p_\t)$. Then,  for a constant $C>0$ independent of $\t$, one has
$$\smallint_0^T\norm{[u_\t-\rdr(v_\t)]_+}^2 + \smallint_0^T\norm{[\rim(v_\t)-u_\t]_+}^2< C\t.$$\label{Pros:L2Epslimit}
\end{proposition}

\begin{proof}
Observe that $\Phi_\tau$ can be rewritten simply as
\begin{equation}
\Phi_\tau(u,v)=-\frac{1}{\t}[u-\rdr(v)]_+ -\frac{1}{\t}[u-\rim(v)]_-.\label{Eq:PhiepsSimple}
\end{equation}
We consider here the case when $|\bm{g}|=0$ and $k(0)=0$, the case $|\bm{g}|\not=0$ and $k(0)>0$ being simpler. From \Cref{Theo:ExistenceEPH}(b), $S_\t$ satisfies $0<S_l\leq S_\t\leq S_r<1$ a.e. in $Q$, where $S_l,\,S_r$ are independent of $\t$. For $\Df(u,v)$, defined as in \eqref{Eq:DefkPsi12}, this implies that there exists $\Df_m=k(S_l)\slash \max_{p\in \R}\{b'(p)\}>0$ such that \ref{ass:D} is satisfied. The pair $(u_\t,v_\t)$ satisfies
\begin{subequations}
\begin{align}
&\smallint_0^T\langle \p_t u_\t, \phi\rangle + \smallint_0^T(\Df(u_\t,v_\t)\del u_{\t}, \del \phi)=\smallint_0^T( \Phi_\tau(u_\t,v_\t),\phi),\label{eq:lastUV1}\\
&\smallint_0^T (\p_t v_\t,\xi)=\smallint_0^T (\Phi_\tau(u_\t,v_\t),\xi),\label{eq:lastUV2}
\end{align}
\end{subequations}
for all $\phi\in L^2(0,T;H^1_0(\Om))$ and $\xi\in L^2(Q)$. Using the test function $\phi=u_\t$ in \eqref{eq:lastUV1} yields
\begin{align}
\frac{1}{2}\norm{u_\t(T)}^2 + \Df_m\smallint_{0}^{T}\norm{\del u_\t}^2 \leq \frac{1}{2}\norm{u_0}^2 + \smallint_{0}^{T} (\Phi_\tau,u_\t).\label{Eq:EpsBoundUIneq}
\end{align}
Using $\xi=\rdr(v_\t)$ in  \eqref{eq:lastUV2} we get
\begin{align}
\smallint^T_0 \smallint_\Om \p_t\left(\smallint^{v_\t}_{V_m} \rdr(\vr)\,d\vr\right)=\smallint^T_0 (\Phi_\tau,\rdr(v_\t)).\label{Eq:EpsBoundVIneq}
\end{align}
Subtracting \eqref{Eq:EpsBoundVIneq} from \eqref{Eq:EpsBoundUIneq}, one further obtains
\begin{align*}
\frac{1}{2}\norm{u_\t (T)}^2 + \Df_m\int_{0}^{T}\norm{\del u_\t}^2 \leq \frac{1}{2}\norm{u_0}^2+  \left\|\smallint^{v_\t(T)}_{v_0} \rdr(\vr)\,d\vr \right\|_{1} + \smallint^T_0(\Phi_\tau,u_\t-\rdr(v_\t)).
\end{align*}
The term $\smallint^{v_\t}_{v_0} \rdr(\vr)\, d\vr$ is bounded in $L^\infty(\Om)$ since both $v_\t$ (from \Cref{Pros:Bound}, $V_m< v_\t < V_M$ a.e. in $Q$) and $\rdr(\cdot)$ are bounded. The last term is estimated as
\begin{align*}
&(\Phi_\tau,u_\t-\rdr(v_\t))=-\frac{1}{\t}\|[u_\t-\rdr(v_\t)]_+\|^2 -\frac{1}{\t}([u_\t-\rim(v_\t)]_-, u_\t-\rdr(v_\t))\\
&\leq -\frac{1}{\t}\|[u_\t-\rdr(v_\t)]_+\|^2 -\frac{1}{\t}([u_\t-\rim(v_\t)]_-, u_\t-\rim(v_\t))\\
&\leq -\frac{1}{\t}\|[u_\t-\rdr(v_\t)]_+\|^2-\frac{1}{\t}\|[u_\t-\rim(v_\t)]_-\|^2.
\end{align*}
Combining everything, we get
\begin{align}
&\frac{1}{2}\norm{u_\t (T)}^2 + \Df_m\smallint_{0}^{T}\norm{\del u_\t}^2 + \frac{1}{\t}\smallint_{0}^{T}\|[u_\t-\rdr(v_\t)]_+\|^2+ \frac{1}{\t}\smallint_{0}^{T}\|[u_\t-\rim(v_\t)]_-\|^2 \nonumber\\
&\leq \frac{1}{2}\norm{u_0}^2+  \left\|\smallint^{v_\t(T)}_{v_0} \rdr(\vr)\, d\vr \right\|_{1}. \label{Eq:EpsBoundFinalIneq}
\end{align}
This proves the assertion. The proof for the case $k(0)>0$ is similar.
\end{proof}

This implies that if a limit $(u,v)$ exists of $(u_\t,v_\t)$ as $\t\to 0$, then $u\in [\rim(v),\rdr(v)]$ in $Q$ implying that the pressure in this limit stays within the bounds of the hysteretic pressure curves.
Another immediate consequence of \Cref{Pros:L2Epslimit} is the boundedness of $\int_{0}^{T}\norm{\del u_\t}^2$ as evident from \eqref{Eq:EpsBoundFinalIneq}. This implies
\begin{corollary}
Under assumptions of \Cref{Pros:L2Epslimit}, there exists $u\in L^2(0,T;H^1_0(\Om))\cap L^\infty(0,T;L^2(\Om))$ such that $u_\t \rightharpoonup u$ in $L^2(0,T;H^1_0(\Om))$ weakly as $\t\to 0$.\label{Corollary:EpsIndBound}
\end{corollary}

\section{Conclusion}
In this paper, we analysed the extended play-type hysteresis (EPH) model, proposed in \cite{Kmitra2017}, for the unsaturated flow case. A modification of the  model is proposed that conforms with the standard weak formulation used in analysis. A regularised version of the problem is considered and it is shown to be equivalent to a nonlinear parabolic equation coupled with an ordinary differential equation. Using Rothe's method, the existence of weak solutions is first proven for the equivalent system and then for the \hyperlink{link:EPH}{(EPH)} model. Solutions are shown to exist even when the capillary pressure functions blow up and relative permeability vanishes for zero saturation. To cover the latter case, the maximum principle needs to hold, which is proven to be the case in the absence of an advection term. The consequences of passing the regularisation parameter to zero are explored,  which show that the pressure stays bounded by capillary pressure curves in the limit.

\section*{Acknowledgment}
The author acknowledges the support from INRIA Paris through the ERC Gatipor grant. The work was also partially carried out in Eindhoven University of Technology under the grant 14CSER016  and in Hasselt University under the grant BOF17BL04. The author would like to thank Prof. Sorin Pop for many fruitful discussions on the topic and his constant encouragement. Also, thanks to the colleagues in Eindhoven for providing valuable comments.

\end{document}